\documentclass[12pt,reqno]{amsart}
\usepackage{amsfonts, amsbsy, amsmath, amssymb, mathtools, amsthm,ragged2e}
\usepackage{amscd}
\usepackage{color,enumerate}
\usepackage{breqn}
\usepackage{bm}
\usepackage{pgfplots}
\pgfplotsset{compat=1.15}
\usepackage{tikz}
\usepackage{tikz-cd}
\usepackage{subcaption}
\usetikzlibrary{arrows}
\usepackage{verbatim}
\usepackage{graphicx}
\usetikzlibrary{matrix,calc}
\usepackage{caption}
\usepackage{hyperref}
\usepackage{float}
\usepackage[mathscr]{euscript}
\usepackage[margin=2.5cm]{geometry}
\usepackage{epsfig}

\newtheorem{thm}{Theorem}[section]
\newtheorem{lem}[thm]{Lemma}
\newtheorem{construction}[thm]{Construction}
\newtheorem{corollary}[thm]{Corollary}

\newtheorem{thm-con}[thm]{Theorem-Conjecture}
\numberwithin{equation}{section}

\theoremstyle{definition}
\newtheorem{defn}[thm]{Definition}
\newtheorem{rmk}[thm]{Remark}

\newtheorem{exmp}[thm]{Example}

\newtheorem{set}[thm]{Set-up}

\DeclareMathOperator{\del}{del}
\DeclareMathOperator{\link}{link}

\DeclareMathOperator{\reg}{reg}
\newcommand{\x}{{\tt X}}

\newcommand{\s}{\mathcal{S}}
\newcommand{\h}{\mathcal{H}}
\newcommand{\p}{\mathcal{P}}
\newcommand{\e}{\mathcal{E}}
\newcommand{\Cla}{\mathcal{C}}
\newcommand{\al}{\alpha}
\newcommand{\ep}{\epsilon}
\newcommand{\et}{\eta}
\newcommand{\de}{\delta}

\begin{document}
	
\title[Vertex cover ideals of simplicial complexes]{Vertex cover ideals of simplicial complexes}
\author{Bijender}
\email{2019rma0011@iitjammu.ac.in}
\address{Department of Mathematics, Indian Institute of Technology Jammu, J\&K, India - 181221.}
\author[Ajay Kumar]{Ajay Kumar}
\email{ajay.kumar@iitjammu.ac.in}
\address{Department of Mathematics, Indian Institute of Technology Jammu, J\&K, India - 181221.}
\date{\today}

\subjclass[2020]{Primary 05E40, 13C14, 13D02}

\keywords{Componentwise linear, vertex decomposable, vertex cover ideal, symbolic powers}

\maketitle
\begin{abstract}
Given a simplicial complex $\Delta$, we investigate how to construct a new simplicial complex $\bar{\Delta}$ such that the corresponding monomial ideals satisfy nice algebraic properties. We give a procedure to check the vertex decomposability of an arbitrary hypergraph. As a consequence,  we prove that attaching non-pure skeletons at all vertices of a cycle cover of a simplicial complex $\Delta$ results in a  simplicial complex $\bar{\Delta}$ such that the associated hypergraph $\h(\bar{\Delta})$ is vertex decomposable. Also, we prove that all symbolic powers of the cover ideal of $\bar{\Delta}$ are componentwise linear. Our work generalizes the earlier known result where non-pure complete graphs were added to all vertices of a cycle cover of a graph.
\end{abstract}

\section{Introduction}
The interplay between commutative algebra and combinatorics has proven to be effective in solving various difficult problems in both areas. In combinatorial commutative algebra, we often deal with squarefree monomial ideals. A squarefree monomial ideal can be viewed as the edge ideal of a simple hypergraph. Let $\mathcal{H}=(V, \mathcal{E})$ be a simple hypergraph with the vertex set $V$ and the edge set $\mathcal{E}$. To a simple hypergraph $\h$, we can associate a squarefree monomial ideal 
$ I(\h)=\langle  \prod\limits_{x \in F}x: F\in \e \rangle \subset R = \mathbb{K}[x_1,\dots,x_n] \rangle$, called the {\it edge ideal} of $\h$. The {\it vertex cover ideal} of $\h$ is defined as $J(\h) = \langle x_{j_1} \cdots x_{j_q}:\{x_{j_1},\dots,x_{j_q}\} ~\mbox{is a minimal vertex cover of}~ \h \rangle.$ It is known that $J(\h)=I(\h)^{\vee}$, where $I(\h)^{\vee}$ denotes the Alexander dual of $I(\h).$ One can also view the edge ideal $I(\h)$ as the facet ideal of a simplicial complex $\Delta$. Let $\Delta$ be a simplicial complex on the vertex set $V=\{x_1,\ldots,x_n\}$ and $\mathcal{F}(\Delta)$ be the set of all facets of $\Delta$. We can associate a squarefree monomial ideal to $\Delta$ given by $I(\Delta)=\langle \prod\limits_{x \in F}x: F \in  \mathcal{F}(\Delta)\rangle$, called the {\it facet ideal} of $\Delta.$ Consider the hypergraph $\mathcal{H}(\Delta)=(V,\mathcal{F}(\Delta))$. Note that $I(\Delta)=I(\h(\Delta)).$ The {\it vertex cover ideal} of $\Delta$ is defined as $J(\Delta)=J(\h(\Delta))=I(\h(\Delta))^{\vee}.$

In this paper, we examine how to modify a simplicial complex $\Delta$ to obtain a new simplicial complex $\bar{\Delta}$ so that $J(\bar{\Delta})$ satisfies nice algebraic properties. Our work is inspired the work of Villareal \cite{vill_cohen}. For any subset $S$ of vertices of a graph $G=(V_G,E_G)$, the graph obtained by attaching a whisker at each vertex of $S$ is denoted by $G \cup W(S)$. Villareal \cite{vill_cohen} showed that $G \cup W(V_G)$ is Cohen-Macaulay or, equivalently, the cover ideal $J(G \cup W(V_G))$ ha a linear resolution. Later in \cite{DAEA,Wood2009}, it was shown that $G \cup W(V_G)$ is vertex decomposable and shellable. In general, for any hypergraph following implications are known:
$$\text{vertex decomposable} \Rightarrow \text{shellable} \Rightarrow \text{sequentially Cohen-Macaulay}.$$
The concept of whiskering and partial whiskering of simplicial complexes has been studied by various authors (see \cite{BJCV,BJVA,CDU,FCH}). 
Authors in \cite{BJVA,FCH} gave some necessary and sufficient conditions on a subset $S$ of the vertex set of a simplicial complex $\Delta$ so that adding a whisker at each vertex of $S$ results in a vertex decomposable or sequentially Cohen-Macaulay simplicial complex. The concept of Alexander duality helps us to link the sequentially Cohen-Macaulay property of a squarefree monomial ideal with componentwise linear property of the dual. Recall that a graded ideal is called {\it componentwise linear,} if for all $j \in \mathbb{N}$, $I_{<j>}$ has a linear resolution, where $I_{<j>}$ denotes the ideal generated by all homogeneous elements of degree $j$ in $I$. For a squarefree monomial ideal $I$, Herzog and Hibi \cite{HHibi} found the following criterion for $I^{\vee}$ to be componentwise linear: $R/I$ is sequentially Cohen-Macaulay over $\mathbb{K}$ if and only if $I^{\vee}$ is a componentwise linear ideal. In particular, this result shows that if a hypergraph $\h$ is vertex decomposable, then $I(\h)^{\vee}$ is a componentwise linear ideal. In \cite{DXTH}, it  has been shown all symbolic powers of the cover ideal of $G \cup W(V_G)$ are componentwise linear. If $S$ is a vertex cover of a graph $G$, then Selvaraja in \cite{2020Selvaraja} proved that all symbolic powers of the cover ideal of $G \cup W(S)$ are componentwise linear. R$\rm{\ddot{o}}$mer in \cite{Tim} proved that a graded $R$-module $M$ is Koszul if and only if $M$ is componentwise
linear. Koszul modules were  introduced Herzog and Iyengar in \cite{HIS}. Koszul modules have significant applications in areas of commutative algebra and algebraic geometry. Authors in \cite{GHSJ} proved that attaching whiskers at all vertices of a cycle cover of $G$ results in a new graph for which all symbolic powers of cover ideals are componentwise linear or equivalently, Koszul. It is natural to ask whether we can extend these results to more general objects. In this paper, we introduce the notion of a cycle cover (see Definition ~\ref{cycle}) of a simplicial complex $\Delta$. We attach non-pure skeletons (see Definition ~\ref{def1}) at each vertex of a cycle cover of a simplicial complex $\Delta$ and obtain a  simplicial complex $\bar{\Delta}$. We generalize the construction given in \cite{KumarAR,Fakhari} to obtain a hypergraph $\h(\ell_1,\ldots,\ell_t)$ from $\h=\h(\bar{\Delta})$ by duplicating the edges in $\h$, where $t$ is the number of edges in $\h$ and $(\ell_1,\ldots,\ell_t)\in \mathbb{N}^t.$ We prove that $\h(\ell,\ldots,\ell)$ is a vertex decomposable graph. Our first main result is stated as follows.
\\\\
{\bf Theorem A} (see Theorem ~\ref{Th-MainResult}).
Let $\Delta$ be a simplicial complex and $W$ be a cycle cover of $\Delta$. Let $\bar{\Delta}$ be the simplicial complex obtained from $\Delta$ by attaching non-pure skeletons at all vertices of $W$.  Then $J(\bar{\Delta})^{(\ell)}$ has linear quotients, and hence it is componentwise linear. \\

If an ideal $I$ has a linear resolution, then  $I^2$ may not have this property. The first such example is due to Terai, which appeared in \cite[Remark 3]{Conca2000}. Another example of such ideals is due to Sturmfels \cite{Sturmfels}. Another important consequence of our work is the following result.\\\\
{\bf Theorem B} (see Theorem ~\ref{TH}).
Let $\Delta$ be a simplicial complex and $W$ be a cycle cover of $\Delta$. Let $\bar{\Delta}$ be the simplicial complex obtained from $\Delta$ by attaching non-pure skeletons at all vertices of $W$. If $J(\bar{\Delta})^{(\ell)}=J(\bar{\Delta})^{\ell}$ for all $\ell \geq 1$, then the following are equivalent.
\begin{enumerate}[(a)]
	\item $J(\bar{\Delta})$ has a linear resolution.
	\item $J(\bar{\Delta})^{\ell}$ has a linear resolution for some $\ell \geq 1$.
	\item $J(\bar{\Delta})^{\ell}$ has a linear resolution for all $\ell \geq 1$.
	\item $R/I(\bar{\Delta})$ is Cohen-Macaulay.
	\item $\bar{\Delta}$ is unmixed.
\end{enumerate}

\section{Preliminaries}
In this section, we introduce some basic notation and terminology used in the paper.
Let $n \in \mathbb{N}$, where $\mathbb{N}$ is the set of all non-negative integers. For simplicity, we denote the set $\{r \in \mathbb{N}_{>0}:r \le n\}$  by $[n].$ Further, we use notation $[n]^a$ for the set 
$\{(\ell_1,\dots,\ell_a): \ell_i \in [n]\}.$

Let $\Delta$ be a simplicial complex on the vertex set $V=\{x_1,\dots,x_n\}.$ If 
$F \in \Delta$, then we say that $F$ is a \emph{face} of $\Delta.$ A face $F$ of $\Delta$ is called a \emph{facet} of $\Delta$ if $F$ is maximal element of $\Delta$ with respect to inclusion. We write $\mathcal{F}(\Delta)$ for the set of all facets of $\Delta.$ We define the \emph{dimension} of a face $F$ of $\Delta$ by $\dim{F} = |F|-1$, where $|F|$ is the cardinality of $F.$ The dimension of the simplicial complex $\Delta$ is defined by
$$\dim{\Delta} = \max\{\dim{F}:F \in \mathcal{F}(\Delta)\}.$$

For $s \in [\dim{\Delta}] \cup \{0\}$, the \emph{$s$th skeleton} of $\Delta$, denoted by $\Delta^{(s)}$, is the simplicial complex defined as
$$\Delta^{(s)} = \{F \in \Delta: \dim{F} \le s\}.$$

If $\Delta$ is a simplicial complex with 
$\mathcal{F}(\Delta) = \{F_1,\dots,F_t\}$, then we say that $\Delta$ is generated by $F_1,\dots,F_t$, and we write
$$\Delta = \langle F_1,\dots,F_t \rangle.$$
Let $\Delta'$ be a simplicial complex such that 
$\mathcal{F}(\Delta') \subset \mathcal{F}(\Delta).$ Then we say that $\Delta'$ is a \emph{sub-collection} of $\Delta.$ By a \emph{simplex}, we mean a simplicial complex with exactly one facet. The concept of pure and non-pure complete graphs is introduced by Selvaraja in \cite{2020Selvaraja}.
For simplicial complexes, we generalize this concept in the following definition.

\begin{defn} \label{def1}
    Let $\Gamma_1,\dots,\Gamma_m$ be simplices with one common vertex $x$ and $\dim{\Gamma_k} \ge 1$ for all $k \in [m].$ For each $k \in [m]$, let 
    $s_k \in [\dim{\Gamma}_k].$ The simplicial complex obtained by joining skeletons $\Gamma_1^{(s_1)},\dots,\Gamma_m^{(s_m)}$ at $x$ is called \emph{skeleton complex} and is denoted by $\Gamma(x).$ In addition to this, $\Gamma(x)$ is said to be a \emph{pure skeleton complex} if for all $k \in [m]$, 
    $s_k < \dim{\Gamma_k}.$ Otherwise, we say that $\Gamma(x)$ is a \emph{non-pure skeleton complex}. 
\end{defn}

\begin{exmp}
    The simplicial complex $\Gamma$ with facets 
    $$\{x_1,x_2,x_3\},\{x_1,x_2,x_4\},\{x_1,x_3,x_4\},\{x_2,x_3,x_4\},\{x_1,x_5,x_6\},\{x_1,x_5,x_7\},\{x_1,x_6,x_7\},\{x_5,x_6,x_7\}$$  as
    shown in Figure ~\ref{Fig-Skeleton1} is an example of a pure skeleton complex while the simplicial complex $\Gamma'$ with facets
    $$\{x_1,x_2,x_3\},\{x_1,x_2,x_4\},\{x_1,x_3,x_4\},\{x_2,x_3,x_4\},\{x_1,x_5,x_6\}$$ as shown in Figure ~\ref{Fig-Skeleton2} is an example of a non-pure skeleton complex. 
    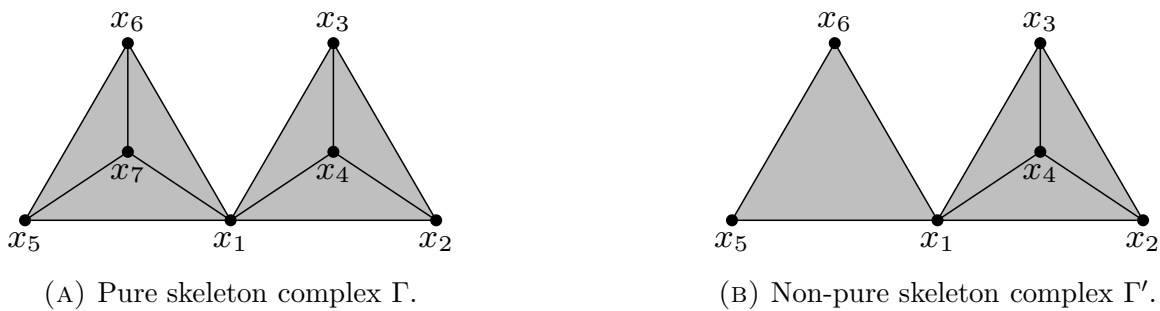
\begin{figure}[H]
	\begin{subfigure}[b]{0.4\textwidth}
		\centering
		\resizebox{\linewidth}{!}{
		\begin{tikzpicture}[scale=0.8][line cap=round,line join=round,>=triangle 45,x=1.0cm,y=1.0cm]
		  \fill[lightgray] (0,0) -- (2.4,0) -- (1.2,2.07) -- cycle;
		  \fill[lightgray] (0,0) -- (-2.4,0) -- (-1.2,2.07) -- cycle;
		  \draw (0,0) -- (2.4,0);
	   	\draw (0,0) -- (1.2,2.07);
    	  \draw (2.4,0) -- (1.2,2.07);
            \draw (1.2,0.8) -- (2.4,0);
		  \draw (1.2,0.8) -- (0,0);
		  \draw (1.2,0.8) -- (1.2,2.07);
		  \draw (0,0) -- (-2.4,0);
		  \draw (0,0) -- (-1.2,2.07);
		  \draw (-2.4,0) -- (-1.2,2.07);
		  \draw (-1.2,0.8) -- (-2.4,0);
		  \draw (-1.2,0.8) -- (0,0);
		  \draw (-1.2,0.8) -- (-1.2,2.07);
				
		  \begin{scriptsize}
				\fill  (0,0) circle (2.0pt);
				\draw[below] (0,0) node {$x_1$};
                \fill  (2.4,0) circle (2.0pt);
				\draw[below] (2.4,0) node {$x_2$};
                \fill  (1.2,2.07) circle (2.0pt);
				\draw[above] (1.2,2.07) node {$x_3$};
				\fill  (1.2,0.8) circle (2.0pt);
				\draw[below] (1.2,0.8) node {$x_4$};
                \fill  (-2.4,0) circle (2.0pt);
			\draw[below] (-2.4,0) node {$x_5$};
                \fill  (-1.2,2.07) circle (2.0pt);
				\draw[above] (-1.2,2.07) node {$x_6$};
				\fill  (-1.2,0.8) circle (2.0pt);
				\draw[below] (-1.2,0.8) node {$x_7$};
		  \end{scriptsize}
		\end{tikzpicture}
		}
		\caption{Pure skeleton complex $\Gamma.$}
        \label{Fig-Skeleton1}
	\end{subfigure}
	\hspace{0.15\textwidth}
	\begin{subfigure}[b]{0.4\textwidth}
		\centering
		\resizebox{\linewidth}{!}{
		\begin{tikzpicture}[scale=0.8][line cap=round,line join=round,>=triangle 45,x=1.0cm,y=1.0cm]
		  \fill[lightgray] (0,0) -- (2.4,0) -- (1.2,2.07) -- cycle;
		  \fill[lightgray] (0,0) -- (-2.4,0) -- (-1.2,2.07) -- cycle;
		  \draw (0,0) -- (2.4,0);
		  \draw (0,0) -- (1.2,2.07);
		  \draw (2.4,0) -- (1.2,2.07);
            \draw (1.2,0.8) -- (2.4,0);
		  \draw (1.2,0.8) -- (0,0);
		  \draw (1.2,0.8) -- (1.2,2.07);
		  \draw (0,0) -- (-2.4,0);
		  \draw (0,0) -- (-1.2,2.07);
		  \draw (-2.4,0) -- (-1.2,2.07);
				
		  \begin{scriptsize}
			\fill  (0,0) circle (2.0pt);
				\draw[below] (0,0) node {$x_1$};
                \fill  (2.4,0) circle (2.0pt);
				\draw[below] (2.4,0) node {$x_2$};
                \fill  (1.2,2.07) circle (2.0pt);
				\draw[above] (1.2,2.07) node {$x_3$};
				\fill  (1.2,0.8) circle (2.0pt);
				\draw[below] (1.2,0.8) node {$x_4$};
                \fill  (-2.4,0) circle (2.0pt);
				\draw[below] (-2.4,0) node {$x_5$};
                \fill  (-1.2,2.07) circle (2.0pt);
				\draw[above] (-1.2,2.07) node {$x_6$};
		  \end{scriptsize}
		\end{tikzpicture}
		}
		\caption{Non-pure skeleton complex $\Gamma'.$}
        \label{Fig-Skeleton2}
	\end{subfigure}
	\caption{Pure and non-pure skeleton complexes.}
	\label{Fig-Skeleton}
    \end{figure}
\end{exmp}

A simplicial complex $\Delta = \langle F_1,\dots,F_t \rangle$ is said to be \emph{connected} if for each $r \neq s$, we have a sequence $F_{k_1},\dots,F_{k_q}$ of facets of $\Delta$ with $F_{k_1} = F_r$ and $F_{k_q} = F_s$ such that $F_{k_p} \cap F_{k_{p+1}} \ne \emptyset$ for all $p \in [q-1].$
A facet $F$ of $\Delta$ is called a \emph{leaf} if it satisfies any one of the following:
\begin{enumerate}[(i)]
    \item $\Delta$ has exactly one facet $F$; or
		
    \item there is a facet $G$ of $\Delta$ with $G \neq F$ such that $F\cap H \subset F\cap G$ for every $H \in \Delta$ with $H \neq F.$
\end{enumerate}
The facet $G$ in (ii) is called a \emph{branch} of F. In addition to this, if we have a linear order $F_{k_1},\dots,F_{k_t}$ of all the facets of $\Delta$ such that $$F \cap F_{k_1} \supset \cdots \supset F \cap F_{k_t},$$ 
then $F$ is called a \emph{good leaf} of $\Delta.$ If every non-empty sub-collection of $\Delta$ has a leaf, then $\Delta$ is called a \emph{forest}. In particular, we use the term \emph{simplicial tree} for a connected forest.
By a \emph{good leaf order}, we mean a linear order $F_1,\dots,F_t$ of all the facets of $\Delta$ such that for every $k \in [t-1]$, $F_k$ is a good leaf of subcollection 
$\langle F_k,\dots,F_t\rangle$ of $\Delta.$ 

Forests are characterized in \cite{Herzog2006SGV} by introducing the notion of special cycles. 
An alternating sequence 
$x_{i_1},F_{i_1},\dots,x_{i_q},F_{i_q},x_{i_{q+1}} = x_{i_1}$ 
of distinct vertices and facets is called a cycle if for each 
$p \in [q]$, $x_{i_p},x_{i_{p+1}} \in F_{i_p}.$ A cycle is special if it has no facet containing more than two vertices of the cycle.

\begin{minipage}{\linewidth}
    \hspace{-0.5cm}	
    \begin{minipage}{0.7\linewidth}
	\begin{exmp}\label{E1}
		  Let $\Delta$ be a simplicial complex with facets $\{x_1,x_2,x_3\},$ 
            $\{x_2,x_4,x_5\},\{x_2,x_3,x_5\},\{x_3,x_5,x_6\}$ as shown in Figure ~\ref{Fig-Cycle}. Then \\
            the cycle $$x_2,\{x_1,x_2,x_3\},x_3,\{x_3,x_5,x_6\},x_5,\{x_2,x_4,x_5\},x_2$$ is special one, while the cycle 
            $$x_2,\{x_1,x_2,x_3\},x_3,\{x_3,x_5,x_6\},x_5,\{x_2,x_3,x_5\},x_2$$ is one which is not special.
		\end{exmp}
    \end{minipage}
    \hspace{-0.4cm}	
    \begin{minipage}{0.22\linewidth}
	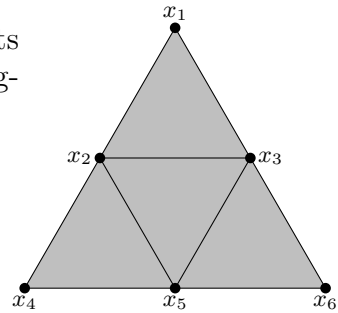
\begin{figure}[H]
            \centering
            \captionsetup{width=1.5\linewidth}
		  \begin{tikzpicture}[node distance=1cm, auto, every node/.style = {inner sep=1pt,outer sep=2pt}, vrtc/.style = {inner sep=3pt}]
				\fill[lightgray] (0,0) -- (2,0) -- (1,1.732) -- cycle;
				\fill[lightgray] (0,0) -- (-2,0) -- (-1,1.732) -- cycle;
				\fill[lightgray] (-1,1.732) -- (0,0) -- (1,1.732) -- cycle;
				\fill[lightgray] (-1,1.732) -- (0,3.464) -- (1,1.732) -- cycle;
				\draw (0,0) -- (2,0);
				\draw (0,0) -- (1,1.732);
				\draw (2,0) -- (1,1.732);
                \draw (0,0) -- (-2,0);
				\draw (0,0) -- (-1,1.732);
				\draw (-2,0) -- (-1,1.732);
				\draw (-1,1.732) -- (0,3.464);
				\draw (-1,1.732) -- (1,1.732);
				\draw (0,3.464) -- (1,1.732);
    
				\begin{scriptsize}
		          \fill  (0,3.464) circle (2.0pt);
                    \draw[above] (0,3.464) node {$x_1$};
		          \fill  (-1,1.732) circle (2.0pt);
		          \draw[left] (-1,1.732) node {$x_2$};
		          \fill  (1,1.732) circle (2.0pt);
		          \draw[right] (1,1.732) node {$x_3$};
		          \fill  (-2,0) circle (2.0pt);
	              \draw[below] (-2,0) node {$x_4$};
		          \fill  (0,0) circle (2.0pt);
		          \draw[below] (0,0) node {$x_5$};
		          \fill  (2,0) circle (2.0pt);
		          \draw[below] (2,0) node {$x_6$};
				\end{scriptsize}
		  \end{tikzpicture}
		  \caption{Cycles and \\
                        special cycles.}
		  \label{Fig-Cycle}
        \end{figure}
    \end{minipage}
\end{minipage}

The following remark gives equivalent conditions for forests.

\begin{thm}\label{Th-Forest}\cite[Theorem 3.2 and Corollary 3.4]{Herzog2006SGV}
	Let $\Delta$ be a simplicial complex. Then
	\begin{enumerate}[(a)]
		\item $\Delta$ is forest if and only if it has no special cycle of length $\ge 3.$
		
		\item $\Delta$ has a good leaf order if and only if $\Delta$ is a forest. In particular, every forest has a good leaf.
	\end{enumerate}
\end{thm}

\begin{defn}\label{cycle}
	Let $\Delta$ be a simplicial complex on the vertex set $V=\{x_1,\dots,x_n\}.$ A subset $W$ of $V$ is called a cycle cover of $\Delta$ if it intersects with every special cycle of $\Delta.$
\end{defn}

It is important to note that if $W$ is a cycle cover of a simplicial complex $\Delta$, then $\Delta \setminus W$ is a forest.
For example, $W = \{x_5\}$ is a cycle cover of the simplicial complex of Figure ~\ref{Fig-Cycle}.

Let $F$ be a face of a simplicial complex $\Delta.$ The simplicial complex defined by the formula
$$\del_{\Delta}(F) = \{ G \in \Delta : G \cap F = \emptyset\}$$
is called the \emph{deletion} of $F$, while the simplicial complex defined by the formula
$$\link_{\Delta}(F) = \{G\in \Delta : G\cap F = \emptyset, G\cup F \in \Delta\}$$ is called the \emph{link} of $F.$
We say that vertex $x$ is a \emph{shedding vertex} of $\Delta$ if for every face $F$ of $\link_{\Delta}(x)$, there exists a face
$F'$ of $\del_{\Delta}(x)$ such that $F \subsetneq F'.$  We say that a simplicial complex $\Delta$ is 
\emph{vertex decomposable} if either it is a simplex or it has a shedding vertex $x$ such that $\del_{\Delta}(x)$ and $\link_{\Delta}(x)$ are vertex decomposable. Furthermore, $\Delta$ is said to be \emph{shellable} if there exists a linear order $F_{k_1}\dots,F_{k_t}$ of all facets of $\Delta$ such that for all $r,s \in [t]$ with $r < s$, there exists $x \in F_{k_s} \setminus F_{k_r}$ and $j \in [s-1]$ with $F_{k_s}\setminus F_{k_j} =\{x\}.$

A \emph{hypergraph} $\h$ on the vertex set $V$ is a pair $\h = (V,\e)$, where $\e$ is a non-empty collection of  subsets of $V.$ Elements of $V$ and $\e$ are called \emph{vertices} and \emph{edges}, respectively. By $V(\h)$, we mean the set of vertices of $\h$, and by $\e(\h)$ we mean the set of edges of $\h.$ We say that a hypergraph $\h$ is \emph{simple} if no edge of $\h$ is contained properly in some other edge. An edge $E$ is called \emph{trivial} if $E = \{x\}$ for some $x \in V.$ Further, a vertex $x \in V$ is called an \emph{isolated} vertex if either $\{x\} \in \e(\h)$, or $x \notin E$ for all $E \in \e(\h).$ An \emph{isolated} hypergraph is one whose every vertex is isolated. In this paper, we only treat simple hypergraphs.

\vspace{0.2cm}
\noindent
\textbf{Notation:}
	Let $\h$ be a hypergraph. Then $\h^\circ$ denotes the hypergraph obtained by deleting isolated vertices of $\h.$

\begin{defn}
    Let $\h = (V(\h),\e(\h))$ be a hypergraph. Then 
    \begin{enumerate}[(i)]
        \item A subset of $V(\h)$ is called a \emph{vertex cover} of $\h$ if it intersects with every edge of $\h.$ A \emph{minimal vertex cover} is a vertex   cover which is minimal among all vertex covers with respect to inclusion.
		
        \item  A subset of $V(\h)$ is called a \emph{independent set} if it does not contains any edge of $\h.$ A \emph{maximal independent set} is an independent set which is maximal among all independent sets with respect to inclusion. 
    \end{enumerate}
\end{defn}

Let $\Delta$ be a simplicial complex on the vertex set $V.$ Then we associate a hypergraph to $\Delta$ given by $\h(\Delta) = (V,\mathcal{F}(\Delta)).$
A subset of $V$ is called a \emph{vertex cover} of $\Delta$ if it is a vertex cover of $\h(\Delta).$
If any two minimal vertex covers of $\Delta$ have the same cardinality, then we say that $\Delta$ is \emph{unmixed}.

The simplicial complex generated by all the maximal independent sets is called \emph{independence complex} of $\h.$ We write $\Delta(\h)$ for the independence complex of a hypergraph $\h.$ By a \emph{shedding vertex} of $\h$, we mean a shedding vertex of $\Delta(\h).$ A hypergraph $\h$ is said to be \emph{vertex decomposable} if its independence complex $\Delta(\h)$ is vertex decomposable.
	
Let $x \in V(\h).$ Consider the sets
$$\e_{\star} = \{E \setminus \{x\}: E \in \e(\h), x \in E\}
$$
and
$$\e^{\star} = \{E \in \e(\h): x \not{\in} E,F\setminus \{x\}\not \subset E ~\mbox{for all} ~ F\in \e(\h)\setminus\{E\}\}.$$	
The \emph{contraction} $\h / x$ of a vertex 
$x \in V(\h)$ in $\h$ is a hypergraph with the vertex set 
$$V(\h / x) = V(\h)\setminus \{x\}$$ and the edge set $$\e(\h / x) = \e_{\star} \cup \e^{\star}.$$
Furthermore, the \emph{deletion} $\h \setminus {x}$ of a vertex 
$x \in V(\h)$ is a hypergraph with the vertex set
$$V(\h \setminus {x}) = V(\h) \setminus\{x\}$$ 
and the edge set
$$\e(\h \setminus {x}) = \{E \in \e(\h): x \not{\in} E\}.$$ 
Observe that if $E \in \e(\h \setminus x)$, then $E \in \e(\h).$ Also, if $E \in \e(\h / x)$, then either $E \in \e(\h)$, or $E \cup \{x\} \in \e(\h).$

\begin{exmp}\label{Ex-NonConstructibleEdge}
    Let $\h$ be a hypergraph on the vertex set $V(\h) = \{x_1,\dots,x_{10}\}$ and the edge set
    $$\e(\h) = \left\{\begin{array}{ll}
                    & \{x_1,x_3,x_5\},\{x_2,x_3,x_5\},\{x_1,x_4,x_5\},\{x_1,x_3,x_6\},\\
                    & \{x_5,x_7,x_9\},\{x_6,x_7,x_9\},\{x_5,x_8,x_9\},\{x_5,x_7,x_{10}\}
	           \end{array}\right\}.$$
    Observe that for $i = 2,4,8,10$, $\{x_i\} \notin \e(\h).$ Therefore, $\{x_i\} \notin \e(\h \setminus x_5)$, however $x_i$ is an isolated vertex in 
    $\h \setminus x_5.$ Also, neither $\{x_6\} \in \e(\h)$, nor $\{x_5,x_6\} \in \e(\h).$ Therefore,  $\{x_6\} \notin \e(\h / x_5)$, however $x_6$ is an isolated vertex in $\h / x_5.$ 
\end{exmp}

We use notation $\h/(x_1,\dots,x_p)$ for the hypergraph $((\h/x_1)\cdots)/x_p$, and $\h \setminus (x_1,\dots,x_p)$ for the hypergraph 
$((\h \setminus x_1) \dots)\setminus x_p.$



Let $\mathbb{K}$ be a field. We identify the vertices of hypergraph $\h$ with variables of the polynomial ring $R = \mathbb{K}[x_1,\dots,x_n].$ The \emph{edge ideal} $I(\h)$ of $\h$ is the squarefree monomial ideal defined by
$$I(\h) = \langle x_{j_1} \cdots x_{j_q}:\{x_{j_1},\dots,x_{j_q}\} \in \e(\h) \rangle.$$
It is important to note that 
$I(\h / x) = I(\h):x$ and $I(\h \setminus {x}) = I(\h)\cap k[x_1,\dots,\hat{x_j},\dots,x_n]$, where $x = x_j \in V(\h).$ Let $I_{\Delta}$ denotes the Stanley–Reisner ideal of a simplicial complex $\Delta.$ Then we have 
$I_{\link_{\Delta}(x)} = (I_{\Delta}:x,x)$, while 
$I_{\del_{\Delta}(x)} = (I_{\Delta},x).$ Since
$I(\h) = I_{\Delta(\h)}$, it follows  immediately that
$\Delta(\h / x) = \link_{\Delta(\h)}(x)$ and $\Delta(\h \setminus {x}) = \del_{\Delta(\h)}(x).$
Thus the definition of vertex decomposability of a hypergraph takes the following form.

\begin{defn}
	A hypergraph $\h = (V(\h),\e(\h))$ is said to be \emph{vertex decomposable} if either it is an isolated hypergraph or it has a shedding vertex $x$ such that $\h/x$ and $H\setminus x$ both are vertex decomposable.
\end{defn}

\noindent
\textbf{Note:} Disjoint union of vertex decomposable hypergraphs is vertex decomposable.

The main purpose of this paper is to study the cover ideals. The \emph{cover ideal} $J(\h)$ of $\h$ is the squarefree monomial ideal given by the formula
$$J(\h) = \langle x_{j_1} \cdots x_{j_q}:\{x_{j_1},\dots,x_{j_q}\} ~\mbox{is a vertex cover of}~ \h \rangle.$$  
One can easily check that $J(\h) = I(\h)^{\lor}$, where $I(\h)^{\lor}$ is the Alexander dual of $I(\h).$ 
The \emph{facet ideal} of a simplicial complex $\Delta$, denoted by $I(\Delta)$, is the squarefree monomial ideal in $R$ given by
$$I(\Delta) = \langle x_{j_1} \cdots x_{j_q}:\{x_{j_1},\dots,x_{j_q}\} \in \mathcal{F}(\Delta) \rangle.$$ 
One can write $I(\Delta) = I(\h(\Delta)).$ Thus, we can define the \emph{vertex cover ideal} of $\Delta$ by 
$J(\Delta) = J(\h(\Delta))$ or equivalently, by $J(\Delta) = I(\Delta)^{\vee}.$ 
    
Let $M$ be a finitely generated $\mathbb{Z}$-graded $R$-module. If 
$\beta_{i,j}^R(M)$ denotes $(i,j)^{th}$ graded Betti number of $M$, then \emph{Castelnuovo-Mumford regularity} of $M$, denoted by $\reg(M)$, is defined as $$\reg(M)=\max\{j-i:\beta_{i,j}^R(M)\neq 0\}.$$
We say that a module $M$ has \emph{linear resolution} if there exists a integer $d$ such that $\beta_{i,i+b}^R(M) = 0$ for all $i$ and for all $b \neq d.$
Further, an $R$-module $M$ is said to be \emph{sequentially Cohen-Macaulay} if there exists a finite sequence
		\begin{equation*}
			0 = M_0\subset M_1 \subset \cdots \subset M_r = M
		\end{equation*}
of graded $R$-submodules of $M$ so that quotient module ${M_i}/{M_{i-1}}$ is Cohen-Macaulay for all $i\in [r]$ and $\dim(M_i/M_{i-1})<\dim(M_{i+1}/M_i)$ for all $i\in [r-1].$ In particular, if $R/I_{\Delta}$ is Cohen-Macaulay (resp. sequentially Cohen-Macaulay) ring, then we say that simplicial complex $\Delta$ is \emph{Cohen-Macaulay} (resp. \emph{sequentially Cohen-Macaulay}) over $\mathbb{K}.$ To say that a hypergraph $\h$ is \emph{shellable} (resp. \emph{sequentially Cohen-Macaulay}) is equivalent to saying that its independence complex $\Delta(\h)$ is shellable (resp. sequentially Cohen-Macaulay). The following implications for a hypergraph are known.
$$\text{Vertex decomposable} \Rightarrow \text{Shellable} \Rightarrow \text{sequentially Cohen-Macaulay}.$$

\begin{defn}
    Let $I$ be a homogeneous ideal of $R.$ For $j \in \mathbb{N}$, let $I_{<j>}$ denotes the ideal generated by all homogeneous elements of $I$ of degree $j.$
    We say that $I$ is \emph{componentwise linear} if $I_{<j>}$ has a linear resolution for all $j.$ 
\end{defn}

The notion of linear quotients is an extremely useful technique to determine whether an ideal is componentwise linear or not. 

\begin{defn} 
   A monomial ideal $I$ of $R$ has \emph{linear quotients} if there exists an ordering $u_1,\dots,u_r$ of minimal generators of $I$ such that the ideal $\langle u_1,\dots,u_{i-1} \rangle: \langle u_i \rangle$ is generated by a subset of $\{x_1,\dots,x_n\}$ for all $2 \le i \le r.$
\end{defn}

Now, we recall the notion of  the $\ell$th symbolic power of a squarefree monomial ideal.

\begin{defn}
    Let $I$ be a squarefree monomial ideal in $R$, and let $I = \wp_1 \cap \cdots \cap \wp_t$ be irredundant primary decomposition of $I$, where $\wp_i$ is an ideal generated by a subset of $\{x_1,\dots,x_n\}.$ Then for $\ell \in \mathbb{N}_{>0}$, the \emph{$\ell$th symbolic power}  $I^{(\ell)}$ of $I$, is defined by
	$$I^{(\ell)} = \wp_1^{\ell} \cap \cdots \cap \wp_t^{\ell}.$$
\end{defn}

We end this section by introducing the notion of polarization. It is a very useful tool to
convert a monomial ideal into a squarefree monomial ideal.

\begin{defn}
    For each $i \in [m]$, let $u_i = \prod_{j = 1}^{n}x_j^{a_{ij}}$ be a monomial in $R$, and let $I = \langle u_1,\dots,u_m \rangle \subset R.$ If 
    $j \in [n]$, then we set $a_j = \max\{a_{ij}:i\in [m]\}.$ Consider a polynomial ring 
    $$T = \mathbb{K}[x_{11},\dots,x_{1a_1},x_{21},\dots,x_{2a_2},\dots,x_{n1},\dots,x_{na_n}].$$ Then the squarefree monomial ideal $\widetilde{I}$ in $T$ generated by squarefree monomials $w_1, \dots ,w_m$, where $w_i = \displaystyle\prod_{j=1}^{n}\prod_{k=1}^{a_{ij}}x_{jk}$, is called \emph{polarization} of $I.$
\end{defn}

\section{Vertex Decomposable Hypergraphs}
In this section, we create some basic tools which helps us to prove the vertex decomposability of a hypergraph.
Let $\ell,a \in \mathbb{N}_{>0}$ and $\mathbf{f} = (f_1,\dots,f_a) \in [\ell]^{a}.$ We write $|\mathbf{f}|$ for the sum $\sum_{p=1}^{a} f_p.$

The following construction introduced in \cite{BAR}, is the main tool for this paper. If $\ell$ is a non-negative integer and $\h$ is a hypergraph, then the Construction ~\ref{Construction} produces a new hypergraph $\h(\ell)$ such that 
 $\widetilde{J(\h)^{(\ell)}} = J(\h(\ell))$. 

\begin{construction}\label{Construction}\rm
    Let $\h$ be a hypergraph with the vertex set $V(\h) = \{x_1,\dots ,x_n\}$ and the edge set $\e(\h) = \{F_1,\dots ,F_t\}$. Suppose 
    $F = \{x_{j_1},\dots ,x_{j_a}\}$ is an edge in $\h.$ We construct a new hypergraph $F(\ell)$ with the vertex set 
    $$V(F(\ell)) =\{x_{j_p,f}: p \in [a], f \in [\ell]\}$$ 
    and the edge set
    $$\e(F(\ell)) =\{\{x_{j_1,f_1},\dots,x_{j_a,f_a}\}:|\mathbf{f}| \le \ell+a-1\},$$ 
    where $\ell \in \mathbb{N}_{>0}.$ Conventionally, $F(0)$ is an isolated hypergraph on the vertex set $$V(F(0))=\{x_{j_1,1},\dots ,x_{j_a,1}\}.$$
    Let $(\ell_1,\dots,\ell_t)\in \mathbb{N}^t$ be an ordered tuple. We construct a new hypergraph $\h(\ell_1,\dots,\ell_t)$ with the vertex set $$V(\h(\ell_1,\dots,\ell_t)) = \bigcup\limits_{i=1}^{t} V(F_i(\ell_i))$$ 
    and the edge set 
    $$\e(\h(\ell_1,\dots,\ell_t)) = \bigcup\limits_{i=1}^{t} \e(F_i(\ell_i)).$$ 
    For the case when $\ell_i = \ell$ for all $i \in [t]$, we denote the hypergraph $\h(\ell_1,\dots,\ell_t)$ by $\h(\ell).$
\end{construction}

The following is an illustrative example for the Construction ~\ref{Construction}.

\begin{exmp}
    Let $\h$ be a hypergraph on the vertex set $V(\h) = \{x_1,x_2,x_3,x_4\}$ and the edge set 
    $\e(\h) = \{\{x_1,x_2,x_3\},\{x_1,x_2,x_4\},\{x_1,x_3,x_4\},\{x_2,x_3,x_4\}\}.$ 
    Then edge set of the hypergraph $\h(1,2,3,2)$ is given by
    $$\e(\h(2)) = \left\{\begin{array}{ll}
                    & \{x_{1,1},x_{2,1},x_{3,1}\},\{x_{1,1},x_{2,1},x_{4,1}\},\{x_{1,2},x_{2,1},x_{4,1}\},\{x_{1,1},x_{2,2},x_{4,1}\},\\
                    & \{x_{1,1},x_{2,1},x_{4,2}\},\{x_{1,1},x_{3,1},x_{4,1}\},\{x_{1,2},x_{3,1},x_{4,1}\},\{x_{1,1},x_{3,2},x_{4,1}\},\\
		          & \{x_{1,1},x_{3,1},x_{4,2}\},\{x_{1,3},x_{3,1},x_{4,1}\},\{x_{1,1},x_{3,3},x_{4,1}\},\{x_{1,1},x_{3,1},x_{4,3}\},\\
                    & \{x_{1,2},x_{3,2},x_{4,1}\},\{x_{1,2},x_{3,1},x_{4,2}\},\{x_{1,1},x_{3,2},x_{4,2}\},\{x_{2,1},x_{3,1},x_{4,1}\},\\
                    & \{x_{2,2},x_{3,1},x_{4,1}\},\{x_{2,1},x_{3,2},x_{4,1}\},\{x_{2,1},x_{3,1},x_{4,2}\}
	           \end{array}\right\}.$$
\end{exmp}

\begin{lem}\label{L-VDmeansCL}
    Let $\h$ be a hypergraph such that $\h(\ell)$ is vertex decomposable; $\ell \in \mathbb{N}.$ Then $J(\h)^{(\ell)}$ has linear quotients, and hence it is componentwise linear.
\end{lem}

\begin{proof}
    Since $\h(\ell)$ is vertex decomposable, we get $\h(\ell)$ is shellable. Now by \cite[Lemma 3.5]{BAR}, we have $\widetilde{J(\h)^{(\ell)}} = J(\h(\ell))$. Thus by \cite[Theorem 8.2.5]{HBook}, $\widetilde{J(\h)^{(\ell)}}$ has linear quotients. By \cite[Lemma 3.5]{Fakhari}, $J(\h)^{(\ell)}$ has linear quotients. Hence by \cite[Theorem 8.2.15]{HBook}, $J(\Delta)^{(\ell)}$ is componentwise linear.  
\end{proof}

We fix some terminology that we use throughout the paper.
Let $\Cla$ denote the collection of all infinite strings with terms in the set $\{D,L\}.$ For any infinite string $\p \in \Cla$, we write $\p_r$ for its $r$th term. If $p \in \mathbb{N}$ and $\p \in \Cla$, then we set $$\Cla(\p,p) = \{\s \in \Cla :\s_r = \p_r ~\mbox{for all}~ r \in [p]\}.$$
Since $[0] = \emptyset$, it follows that $\Cla(\p,0) = \Cla$ for all $\p \in \Cla.$

Let $\h$ be a hypergraph on the vertex set $V =\{x_1,\dots,x_n\}$ and the edge set $\e(\h) = \{F_1,\dots,F_t\}$, and 
$(\ell_1,\dots,\ell_t) \in \mathbb{N}^t_{>0}.$ We write $\bm{x}_p$ for the $p$th term of a sequence  $\bm{x}.$ Let $\s$ be an infinite string in $\Cla$ and $\bm x = (x_{i_r,c_r})_{r = 1}^{\al}$ be a sequence of distinct vertices in $V(\h(\ell_1,\dots,\ell_t).$ We define a sequence of hypergraphs $(\h[\s,\bm x;r])_{r = 0}^{\infty}$, recursively as follows: Define $$\h[\s,\bm x;0] = \h(\ell_1,\dots,\ell_t).$$ 
Suppose $r \ge 1$ and $\h[\s,\bm x;r-1]$ is defined. Define
    \[ 
    \h[\s,\bm x;r] =\begin{cases}
                            \h[\s,\bm x;r-1]/{x_{i_r,c_r}} & ~\mbox{for}~ r \in [\al] ~\mbox{and}~ \s_r = L; \\
                            \h[\s,\bm x;r-1]\setminus {x_{i_r,c_r}} &  ~\mbox{for}~ r \in [\al] ~\mbox{and}~ \s_r = D;\\
                            \h[\s,\bm x;\al] & ~\mbox{for}~ r > \al.
	                \end{cases}
    \]
The hypergraph $\h[\s,\bm x;r]$ is called the \emph{$r$th hypergraph} determined by the string $\s$ and the sequence $\bm x.$

We say that the sequence $\bm x$ satisfies \emph{property $\mathfrak{P}$} with respect to $\s$ if for all $\de, \ep \in [\al]$ with $\de < \ep$ and 
$\s_{\de} = \s_{\ep} = L$, we have $x_{i_{\de}} \ne x_{i_{\ep}}.$ Let $\bm x$ satisfies the property $\mathfrak{P}$ with respect to $\s$ and $r \ge 0.$ We define
$$A[\s,\bm x;r] = \{x_{i_p} \in V(\h):x_{i_p,c_p} = {\bm x}_p, \s_p = L,p \in  [\min\{r,\al\}]\}$$ and 
$$B[\s,\bm x;r] = \{x_{i_p} \in V(\h):x_{i_p,c_p} = {\bm x}_p, \s_p = D,p \in [\min\{r,\al\}]\}.$$
Then, the set $A[\s,\bm x;r]$ is called \emph{$r$th set of vertices of contraction} determined by the string $\s$ and the sequence $\bm x$, and the set 
$B[\s,\bm x;r]$ is called  \emph{$r$th set of vertices of deletion} determined by the string $\s$ and the sequence $\bm x.$
Observe that $A[\s,\bm x;r] = A[\s,\bm x;\al]$ and $B[\s,\bm x;r] = B[\s,\bm x;\al]$ for all $r \ge \al.$
Whenever, string $\s$ and sequence $\bm x$ is clear from the context, we use notation $A_r$ for the set $A[\s,\bm x;r]$, and $B_r$ for the set 
$B[\s,\bm x;r].$
Further, if $r \in \mathbb{N}$ and $k \in [t]$, then we set 
$\ell_{k,r} = \max\{0,\ell_k - d_{k,r}\}$, where
$$d_{k,r} = -|F_k \cap A_r|+\sum\limits_{\substack{x_{i_p} \in F_k \cap A_r,\\x_{i_p,c_p} = {\bm x}_{p}}} c_p.$$
If $r \in [\al]$ with $\s_r = L$, then $x_{i_r} \notin A_{r-1}.$ Thus each $\ell_{k,r}$ is well-defined.
        
We say that $E = \{x_{j_1,f_1},\dots,x_{j_a,f_a}\}\subset V(\h(\ell_1,\dots,\ell_t))$ is \emph{constructible} in $\h[\s,\bm x;r]$ if there exists 
$k \in [t]$ such that the following conditions are satisfied:
\begin{enumerate}[(i)]
    \item $\{x_{j_1},\dots,x_{j_a}\} = F_k \setminus A_r$; and
		
    \item $\mathbf{f} = (f_1,\dots,f_a) \in [\ell_k]^a$ with $f_q > c_p$ whenever $x_{j_q} = x_{i_p} \in B_r\setminus A_r$ and $|\mathbf{f}| \le \ell_{k,r}+a-1.$
\end{enumerate}

\begin{exmp}
    Let $\h$ be a hypergraph on the vertex set $V = \{x_1,x_2,x_3,x_4\}$ and the edge set 
    $\e(\h) = \{F_1 = \{x_1,x_2,x_3\},F_2 = \{x_1,x_2,x_4\},F_3 = \{x_1,x_3,x_4\},F_4 = \{x_2,x_3,x_4\}\}.$ Let $\s$ be an infinite string in $\Cla$ such that 
    $\s_1 = \s_3 = D$ and $\s_2 = L.$ Consider the sequence $\bm x = (x_{1,1},x_{1,2},x_{2,1})$ of vertices in $V(\h(2)).$ We obtain 
    $$\e((\h[\s,\bm x;1])^\circ) =
    \left\{\begin{array}{ll}
	                   & \{x_{1,2},x_{2,1},x_{3,1}\},\{x_{1,2},x_{2,1},x_{4,1}\},\{x_{1,2},x_{3,1},x_{4,1}\},\{x_{2,1},x_{3,1},x_{4,1}\},\\
	                   & \{x_{2,2},x_{3,1},x_{4,1}\},\{x_{2,1},x_{3,2},x_{4,1}\},\{x_{2,1},x_{3,1},x_{4,2}\}
    \end{array}\right\}$$ 
    
    and $$\e((\h[\s,\bm x;2])^\circ) = \Bigl\{\{x_{2,1},x_{3,1}\},\{x_{2,1},x_{4,1}\},\{x_{3,1},x_{4,1}\}\Bigr\}.$$
    Also, for each $r \ge 3$, we have $$\e((\h[\s,\bm x;r])^\circ) = \Bigl\{\{x_{3,1},x_{4,1}\}\Bigr\}.$$
    It is easy to see that the set $\{x_{2,2},x_{3,1},x_{4,1}\}$ is constructible in $\h[\s,\bm x;3]$ but $\{x_{2,1},x_{3,2},x_{4,1}\}$ is not.
    
\end{exmp}

We repeatedly use the following lemma throughout the paper.

\begin{lem}\label{L-VDtool}
    Let $\h$ be a hypergraph and $(\ell_1,\dots,\ell_t) \in \mathbb{N}^t_{>0}$, where $t$ is the number of edges of $\h.$
    Further, let $p \in \mathbb{N}$ and $\p$ is any infinite string. Suppose that for each $\s \in \Cla(\p,p)$, there exists a sequence 
    $\bm x(\s)$ of vertices in $\h(\ell_1,\dots,\ell_t)$ with $\al(\s)$ terms such that
    \begin{enumerate}\rm
        \item \textit{$\al(\s) \ge p$,}
        
		\item \textit{$\h[\s,\bm x(\s);\al(\s)]$ is a vertex decomposable hypergraph,}

        \item \textit{if $\al(\s) > p$ and $p \le p' < \al(\s)$, then for each $\s' \in \Cla(\s,p')$, we have 
        $p' < \al(\s')$ and ${\bm x(\s)}_r = {\bm x(\s')}_r$ for all $r \in [p'+1]$, and}
		
		\item \textit{${\bm x(\s)}_{r}$ is a shedding vertex of $(\h[\s,\bm x(\s);r-1])^\circ$ for all 
		$p < r \le \al(\s).$} 
	\end{enumerate}
    Then $\h[\s,\bm x(\s);r]$ is a vertex decomposable hypergraph for all $r \ge p$ and for all $\s \in \Cla(\p,p).$ 
\end{lem}

\begin{proof}
    Let $\al_0 = \max\{\al(\s):\s \in \Cla(\p,p)\}.$ Then $\h[\s,\bm x(\s);r] = \h[\s,\bm x(\s);\al(\s)]$ for all $r \ge \al_0$ and for all $\s \in \Cla(\p,p).$ Therefore, we assume that $p \le r \le \al_0.$ We proceed by using induction on $\al_0-r.$ The result holds for $r = \al_0.$ Now, suppose that 
    $p \le r < \al_0$ and $\s$ be any infinite string in $\Cla(\p,p)$. If $r \ge \al(\s)$, then 
    $$\h[\s,\bm x(\s);r] = \h[\s,\bm x(\s);r+1],$$ and hence the result follows from induction. On the other hand, suppose that $r \in [\al(\s)-1] \cup \{0\}$ and $\s_{r+1} = D.$ Then $$\left(((\h[\s,\bm x(\s);r])^\circ)\setminus {{\bm x(\s)}_{r+1}}\right)^\circ = (\h[\s,\bm x(\s);r+1])^\circ.$$ Now, consider the string 
    $\s'$ with $\s'_s = \s_s$ for all $s \ne r+1$ and $\s'_{r+1} = L.$ Then $\s' \in \Cla(\s,r).$ Therefore, by given hypothesis, we have
    $$\left(((\h[\s,\bm x(\s);r])^\circ)/{\bm x(\s)}_{r+1} \right) ^\circ  = (\h[\s',\bm x(\s');r+1])^\circ.$$ Thus by induction, 
    $(\h[\s,\bm x(\s);r])^\circ)\setminus {{\bm x(\s)}_{r+1}}$ and $((\h[\s,\bm x(\s);r])^\circ)/{{\bm x(\s)}_{r+1}}$ are vertex decomposable hypergraphs. 
    Since ${\bm x(\s)}_{r+1}$ is a shedding vertex of the hypergraph $(\h[\s,\bm x(\s);r])^\circ$, we get $\h[\s,\bm x(\s);r]$ is vertex decomposable hypergraph. In a similar manner, we can prove the result when $\s_{r+1} = L.$
\end{proof}

\begin{rmk}\label{R-BasicRemark}
    Let $\h$ be a hypergraph, $\s$ be an infinite string in $\Cla$ and $(\ell_1,\dots,\ell_t) \in \mathbb{N}^t_{>0}$, where $t$ is the number of edges of $\h.$
    Further, let $\bm x = (x_{i_r,c_r})_{r = 1}^{\al}$ be a sequence of vertices in $V(\h(\ell_1,\dots,\ell_t))$ that satisfies the property $\mathfrak{P}$ with respect to $\s$ and $c_r = \min\{c:x_{i_r,c} \in V((\h[\s,\bm x;r-1])^\circ)\}.$ 
    Suppose that $q \in [\al].$ 
    \begin{enumerate}[(a)]
		\item Let $x_{i_q} = x_{i_p}$ for some $p \in [q-1]$ with $\s_p = D.$ Note that
        \begin{equation}\label{EQ1}
            V((\h[\s,\bm x;q-1])^\circ) \subset V((\h[\s,\bm x;p'])^\circ)
        \end{equation}
        for all $0 \le p' \le q-1.$ 
        It follows from ~\eqref{EQ1} that $x_{i_q,c_q} \in V((\h[\s,\bm x;p-1])^\circ).$ Thus, we have $c_{q} \ge c_p.$ Also, note that 
        $$\h[\s,\bm x;p] = \h[\s,\bm x;p-1] \setminus x_{i_{p},c_{p}}.$$ This forces $x_{i_{p},c_{p}} \notin V((\h[\s,\bm x;p])^\circ).$ Therefore, in view of \eqref{EQ1}, we can say that $c_{q} > c_p.$
		
		\item Let $E$ be an edge in $\h[\s,\bm x;q-1]$ with $x_{i_{q},c_{q}} \in E.$ If $E$ is a trivial edge, then 
        $E = \{x_{i_q,c_q}\}.$ Since $x_{i_q,c_q} \in V((\h[\s,\bm x;q-1])^\circ)$, there exists an edge $E' \in \e((\h[\s,\bm x;q-1])^\circ)$ such that $x_{i_q,c_q} \in E'.$ Then $E \subsetneq E'$, which contradicts the fact that $E' \in \e((\h[\s,\bm x;q-1])^\circ).$ Thus we must have 
         $E \in \e((\h[\s,\bm x;q-1])^\circ).$
    \end{enumerate}
\end{rmk}

The following theorem is proved in \cite{BAR} for the special case when 
$\h = \h(\Delta)$, where $\Delta$ is a simplicial tree. One can check that the proof does not require properties of a simplicial tree and can be proved for an arbitrary hypergraph. Here, we state the theorem for an arbitrary hypergraph and avoid the proof. 

\begin{thm}\label{Th-Constructible}
    With the notation of Remark ~\ref{R-BasicRemark}, any edge $E$ in $\h[\s,\bm x;r]$ is constructible in $\h[\s,\bm x;r].$ 
    Conversely, for every constructible set $E$ in $\h[\s,\bm x;r]$, there is a edge $E'$ in $\h[\s,\bm x;r]$ contained in $E.$
\end{thm}

In view of Theorem ~\ref{Th-Constructible}, if $E = \{x_{j,f}\}$ is constructible in $\h[\s,\bm x;r]$ if and only if $E$ is an edge in $\h[\s,\bm x;r].$ 
Let $\h$ be a hypergraph on the vertex set $V = \{x_1,\dots,x_5\}$ and edge set $\e(\h) = \{\{x_1,x_2,x_3\},\{x_3,x_4,x_5\}\}.$ Note that 
$\{x_{1,2}\} \notin \e(\h(2) \setminus x_{3,1}).$ Therefore $\{x_{1,2}\}$ is not constructible in $\h(2) \setminus x_{3,1}$ (see Example ~\ref{Ex-NonConstructibleEdge}).

\begin{corollary}\label{C-Constructible}
    With the notation of Remark ~\ref{R-BasicRemark}, let $r \ge 0$ and $x_{j,f} \in V((\h[\s,\bm x;r])^\circ).$ Then $x_j \notin A_r.$ Also, if 
    $x_j = x_{i_p} \in B_r \setminus A_r$, then $f > c_p.$
\end{corollary}

\begin{proof}
    Since $x_{j,f} \in V((\h[\s,\bm x;r])^\circ)$, there exists an edge $E \in \e((\h[\s,\bm x;r])^\circ)$ such that $x_{j,f} \in E.$ Now, by using Theorem ~\ref{Th-Constructible}, $E$ is constructible in $\h[\s,\bm x;r].$ Thus we obtain the desired result.
\end{proof}

\section{Non-Pure skeleton complexes and componentwise linear ideals}

In this section, we prove the first main result of this article that if $\bar{\Delta}$ is a simplicial complex obtained from a simplicial complex $\Delta$ by attaching non-pure skeleton complexes to each vertex in a cycle cover of $\Delta$, then for each 
$\ell \in \mathbb{N}$, $J(\bar{\Delta})^{(\ell)}$ is componentwise linear. 
We start with the following definition.

\begin{defn}
    Let $\h$ be a hypergraph and $(\ell_1,\dots,\ell_t) \in \mathbb{N}^t_{>0}$, where $t$ is the number of edges in $\h.$ 
    \begin{enumerate}[(i)]
        \item The \emph{order induced by shadows} of vertices on $\h(\ell_1,\dots,\ell_t)$, denoted by `$\prec_s$', is the total order on $V(\h(\ell_1,\dots,\ell_t))$ defined by setting $x_{i,c} \prec_s x_{j,d}$ if either  (i) $c < d$, or (ii) $c = d$ and $i <j.$

        \item The \emph{order induced by index} of vertices on $\h(\ell_1,\dots,\ell_t)$, denoted by `$\prec_v$', is the total order on $V(\h(\ell_1,\dots,\ell_t))$ defined by setting
    $x_{i,c} \prec_v x_{j,d}$ if either  (i) $i < j$, or (ii) $i = j$ and $c < d.$ 
    \end{enumerate}
\end{defn}

Since $V(\h(\ell_1,\dots,\ell_t))$ is a finite set, it follows that $\prec_s$ and $\prec_v$ are well-ordering on $V(\h(\ell_1,\dots,\ell_t)).$ 

\begin{thm}\label{Th-Skeleton}
    Let $\Gamma$ be a simplex on the vertex set $V = \{x_1,\dots,x_n\}$ and $\h = \h(\Gamma^{(s)})$, where $s \in [n-1] \cup\{0\}.$ Then the hypergraph $\h(\ell)$ is vertex decomposable for all $\ell \in \mathbb{N}_{>0}.$
\end{thm}

\begin{proof}
    For $s = 0$, $\h(\ell)$ is an isolated hypergraph, and hence it is vertex decomposable. Assume that $s \ge 1.$
    Let $\Gamma^{(s)} = \langle F_1,\dots,F_t \rangle$ and $(\ell_1,\dots,\ell_t) \in \mathbb{N}^t_{>0}$ be an ordered tuple with $\ell_k = \ell$ for all 
    $k \in [t].$ Suppose that $\s$ is an infinite string in $\Cla.$ 
    We define a sequence $(\bm x(\s,p) )_{p = 1}^{\infty}$, where $\bm x(\s,p)$ is a finite sequence of vertices in $\h(\ell)$ that satisfies the property $\mathfrak{P}$ with respect to $\s$, recursively as follows:
    For $p = 1$, define $$\bm x(\s,1) = (x_{1,1}).$$
    Now, let $p \ge 1$ be an integer and sequence $\bm x(\s,p)$ that satisfies the property $\mathfrak{P}$ with respect to $\s$ is defined. If  
    $V((\h[\s,\bm x(\s,p);p])^\circ) \ne \emptyset$, then we define 
    $\bm x(\s,p+1)$ as
    $$
    \begin{array}{ccl}
        {\bm x(\s,p+1)}_q & = &
        \left\{\begin{array}{lll}
        {\bm x(\s,p)}_q & ~\mbox{if}~ q \in [p];\\
        \min{(V((\h[\s,\bm x(\s,p);p])^\circ),\prec_s)} & ~\mbox{if}~ q = p+1.
        \end{array}\right.
    \end{array}
    $$
    On the other hand, if $V((\h[\s,\bm x(\s,p);p])^\circ) = \emptyset$, then we define $\bm x(\s,p') = \bm x(\s,p)$ for all $p' > p.$
    To see that $x(\s,p+1)$ satisfies the property $\mathfrak{P}$ with respect to $\s$, let $\de,\ep \in [p+1]$ with $\de < \ep$ and $\s_{\de} = \s_{\ep} = L.$ If 
    $x(\s,p+1) = x(\s,p)$, or $\ep < p+1$, then we are done. Now, let $x(\s,p+1) \ne x(\s,p)$ and $\ep = p+1.$ If $x(\s,p+1)_{\ep} = x_{j,f}$, then it follows from Corollary ~\ref{C-Constructible} that $x_j \notin A[\s,\bm x(\s,p);p]$, and hence $x(\s,p+1)$ satisfies the property $\mathfrak{P}$ with respect to $\s.$   
    
    Let $\al(\s)$ be the smallest integer such that $V((\h[\s,\bm x(\s,\al(\s));\al(\s)])^\circ) = \emptyset.$ Now, we set
    $$\bm x(\s) = \bm x(\s,\al(\s)).$$
    Note that $\bm x(\s)$ has $\al(\s)$ terms and it satisfies the property $\mathfrak{P}$ with respect to $\s.$ 
    We prove that the sequence $\bm x(\s)$ satisfies the hypothesis of Lemma ~\ref{L-VDtool} for $p = 0.$ Clearly, $\al(\s) \ge 0.$ Note that 
    $\h[\s,\bm x(\s);\al(\s)]$ is an isolated hypergraph since $V((\h[\s,\bm x(\s);\al(\s)])^\circ) = \emptyset.$ Thus $\h[\s,\bm x(\s);\al(\s)]$ is vertex decomposable. Let $0 \le p' < \al(\s)$ and $\s' \in \Cla(\s,p').$ If $\al(\s') \le p'$, then 
    $\bm x(\s',\al(\s')) = \bm x(\s,\al(\s'))$, and hence 
    $$\h[\s',\bm x(\s',\al(\s'));\al(\s')] = \h[\s,\bm x(\s,\al(\s'));\al(\s')].$$ 
    Now, the fact $V((\h[\s',\bm x(\s',\al(\s'));\al(\s')])^\circ) = \emptyset$ implies that $\al(\s') \ge \al(\s) > p'$, a contradiction. 
    Thus $p' <  \al(\s')$ for all $\s' \in \Cla(\s,p').$ Also, we have ${\bm x(\s)}_r = {\bm x(\s')}_r$ for all $r \in [p'+1].$
    
    Now, it is remains to prove that $\bm x(\s)_r$ is a shedding vertex of $(\h[\s,\bm x(\s);r-1])^\circ$ for all
    $r \in [\al(\s)].$ Note that  
    $$\bm x(\s)_r = \min{(V((\h[\s,\bm x(\s);r-1])^\circ),\prec_s)}.$$ 
    Let $r \in [\al(\s)]$ and $U \in \link_{\mathcal{D}}(\bm x(\s)_r)$, where $\mathcal{D} = \Delta((\h[\s,\bm x(\s);r-1])^\circ).$ 
    Suppose that $\bm x(\s)_r = x_{i_r,c_r}$ and $F \in \e((\h[\s,\bm x(\s);r-1])^\circ)$ with $x_{i_r,c_r} \in F.$ Then there exists $j \in [n]$ with $j \ne i_r$ such that $x_{j,c} \in F \setminus U$ for some $c \in [\ell].$ Let $f$ be the largest integer such that 
    $x_{j,f} \in V((\h[\s,\bm x(\s);r-1])^\circ) \setminus U.$ We claim that $U \cup \{x_{j,f}\} \in \mathcal{D}.$ On the contrary, suppose 
    $U \cup \{x_{j,f}\} \notin \mathcal{D}.$ Then there exists an edge $E = \{x_{j_1,f_1},\dots,x_{j_a,f_a}\} \in \e((\h[\s,\bm x(\s);r-1])^\circ)$ such that 
    $E \subset U \cup \{x_{j,f}\}.$ Since $U \in \mathcal{D}$, we see that $x_{j,f} \in E.$ Without loss of generality, we assume that $x_{j,f} = x_{j_1,f_1}.$ Write $A_{r-1} = A[\s,\bm x(\s);r-1]$ and $B_{r-1} = B[\s,\bm x(\s);r-1].$ Further, if $k \in [t]$, then we set $\ell_{k,r-1} = \max\{0,\ell_k - d_{k,r-1}\}$, where
    $$d_{k,r-1} = -|F_k \cap A_{r-1}|+\sum\limits_{\substack{x_{i_p} \in F_k \cap A_{r-1},\\x_{i_p,c_p} = {\bm x}_{p}}} c_p.$$
    Using Theorem ~\ref{Th-Constructible} that there exists $k \in [t]$ such that $\{x_{j_1},\dots,x_{j_a}\} = F_k \setminus A_{r-1}$ and $\mathbf{f} = (f_1,\dots,f_a) \in [\ell]^a$ with 
    $$f_q > c_p ~\mbox{if}~ x_{j_q} = x_{i_p} \in B_{r-1} \setminus A_{r-1} ~\mbox{and}~ |\mathbf{f}| \le \ell_{k,r-1}+a-1.$$ 
    Now, we have the following two cases.\\
	
    \noindent
    \textbf{Case 1.} When $x_{i_r} \in F_k.$ Without loss of generality, assume that $x_{i_r} = x_{j_2}.$ Then $f_2 > c_r \ge 1.$ Since
    $|\mathbf{f}| \le \ell_{k,r-1}+a-1 \le \ell+a-1$, we get $f = f_1 < \ell.$ Therefore, the set 
    $$\bar{E} = \{x_{j_1,f_1+1},x_{j_2,c_r},\dots,x_{j_a,f_a}\} \subset V(\h(\ell)).$$ Now, since 
    $|\mathbf{f}|-f_2+c_r+1 \le |\mathbf{f}| \le \ell_{k,r-1}+a-1$, it follows from Remark ~\ref{R-BasicRemark}(a) that $\bar{E}$ is constructible in 
    $\h[\s,\bm x(\s);r-1].$ By Theorem ~\ref{Th-Constructible}, there exists an edge $E'$ in $\h[\s,\bm x(\s);r-1]$ with $E' \subset \bar{E}.$ 
    Since $E \in \e(\h[\s,\bm x(\s);r-1])$, we must have $E' \cap \{x_{j_1,f_1+1},x_{j_2,c_r}\} \ne \emptyset.$ Suppose
    $E' \notin \e((\h[\s,\bm x(\s);r-1])^\circ).$ By Remark ~\ref{R-BasicRemark}(b), we get $E' = \{x_{j_1,f_1+1}\}.$ The fact that  $E'$ is a edge in $\h[\s,\bm x(\s);r-1]$ gives that $\{x_{j_1,f_1}\}$ is constructible in $\h[\s,\bm x(\s);r-1]$. Now by using Theorem ~\ref{Th-Constructible}, we get $\{x_{j_1,f_1}\}$ is an edge in $\h[\s,\bm x(\s);r-1].$ Therefore $E \notin \e((\h[\s,\bm x(\s);r-1])^\circ)$, a contradiction. Thus, we have 
    $E' \in \e((\h[\s,\bm x(\s);r-1])^\circ).$ Now, by maximality of $f$, we have $E' \subset U \cup \{x_{i_r,c_r}\}.$ This contradicts the fact that 
    $U \in \link_{\mathcal{D}}(\bm x(\s)_r).$\\
	
    \noindent
    \textbf{Case 2.} When $x_{i_r} \notin F_k.$ Since $\Gamma$ is a simplex, it follows that $(F_k \setminus \{x_{j_1}\}) \cup \{x_{i_r}\} = F_{k'}$ for some 
    $k' \in [t].$ In view of Corollary ~\ref{C-Constructible}, we have $F_{k'} \cap A_{r-1} = F_k \cap A_{r-1}$ and 
    $$F_{k'} \setminus A_{r-1} = \{x_{i_r},x_{j_2},\dots,x_{j_a}\}.$$
    Therefore, $\ell_{k',r-1} = \ell_{k,r-1}.$ Note that $f_1 \ge c_r.$ Thus $|\mathbf{f}|-f_1+c_r \le \ell_{k',r-1}+a-1$, and hence by Remark ~\ref{R-BasicRemark}(a),
    $\bar{E} = \{x_{i_r,c_r},x_{j_2,f_2},\dots,x_{j_a,f_a}\}$ is constructible in $\h[\s,\bm x(\s);r-1].$ By Theorem ~\ref{Th-Constructible}, there exists an edge $E'$ in $\h[\s,\bm x(\s);r-1]$ with $E' \subset \bar{E}.$ Since $E \in \e((\h[\s,\bm x(\s);r-1])^\circ)$, we get $x_{i_r,c_r} \in E'.$ In view of the Remark 
    ~\ref{R-BasicRemark}(b), we have $E' \in \e((\h[\s,\bm x(\s);r-1])^\circ).$ Now, the fact $E' \subset U \cup \{x_{i_r,c_r}\}$ gives us a contradiction.
    
    By above two cases, we get $U \cup \{x_{j,f}\} \in \mathcal{D}.$ Since $i_r \ne j$, we conclude that 
    $U \cup \{x_{j,f}\} \in \del_{\mathcal{D}}(x_{i_r,c_r}).$ Therefore, $x_{i_r,c_r}$ is a shedding vertex of 
    $(\h[\s,\bm x(\s);r-1])^\circ.$ This proves the theorem.
\end{proof}

Let $\Gamma$ be a simplex on the vertex set $V = \{x_1,x_2,x_3,x_4\}$ and $\h = \h(\Gamma^{(2)}).$ Then, as seen in the proof of Theorem ~\ref{Th-Skeleton}, $x_{1,1}$ is a shedding vertex of $\h(2).$ It is easy to see that $x_{1,2}$ is not a shedding vertex of $\h(2)\setminus x_{1,1}$ but hypergraphs 
$\h(2)\setminus (x_{1,1},x_{1,2})$ and $(\h(2)\setminus x_{1,1})/ x_{1,2}$ are vertex decomposable. This leads us to prove the following useful result.  

\begin{thm}\label{Th-GSkeleton}
    Let $\Gamma$ be a simplex on the vertex set $V = \{x_1,\dots,x_n\}$ and $\h = \h(\Gamma^{(s)})$, where $s \in [n-1].$ Suppose $\ell \in \mathbb{N}_{>0}$ and $\nu \in [\ell-1].$ Then the hypergraphs $\h(\ell)/x_{1,1}$, $(\h(\ell) \setminus (x_{1,1},\dots,x_{1,\nu}))/x_{1,\nu+1}$ and 
    $\h(\ell) \setminus (x_{1,1},\dots,x_{1,\ell})$ are vertex decomposable.
\end{thm}

\begin{proof}
    Let $\Gamma'$ be a simplex on the vertex set $\{x_2,\dots,x_n\}$, $\h' = \h({\Gamma'}^{(s-1)})$ and $\h'' = \h({\Gamma'}^{(s)}).$ Then 
    $(\h(\ell)/x_{1,1})^\circ = (\h'(\ell))^\circ$ and $(\h(\ell) \setminus (x_{1,1},\dots,x_{1,\ell}))^\circ \simeq (\h''(\ell))^\circ.$ By Theorem ~\ref{Th-Skeleton}, it follows that the hypergraphs $\h(\ell)/x_{1,1}$ and $\h(\ell) \setminus (x_{1,1},\dots,x_{1,\ell})$ are vertex decomposable.
    Now, we prove that $(\h(\ell) \setminus (x_{1,1},\dots,x_{1,\nu}))/x_{1,\nu+1}$ is a vertex decomposable hypergraph.
    Let $\Gamma^{(s)} = \langle F_1,\dots,F_t \rangle$ and $(\ell_1,\dots,\ell_t) \in \mathbb{N}^t_{>0}$ be an ordered tuple with $\ell_k = \ell$ for all 
    $k \in [t].$ Suppose that $\p$ is an infinite string in $\Cla$ with $\p_r = D$ for $r \in [\nu]$ and $\p_{\nu+1} = L$ and $\s \in \Cla(\p,\nu+1).$ 
    We define a sequence $(\bm x(\s,p) )_{p = 1}^{\infty}$, where $\bm x(\s,p)$ is a finite sequence of vertices in $\h(\ell)$ that satisfies the property $\mathfrak{P}$ with respect to $\s$, recursively as follows:
    For $p = 1$, define	$$\bm x(\s,1) = (x_{1,1},\dots,x_{1,\nu},x_{1,\nu+1}).$$
    Now, let $p \ge 1$ be an integer and sequence $\bm x(\s,p)$ that satisfies the property $\mathfrak{P}$ with respect to $\s$ is defined. If 
    $V((\h[\s,\bm x(\s,p);\nu+p])^\circ) \ne \emptyset$, then we define 
    $\bm x(\s,p+1)$ as
    $$
    \begin{array}{ccl}
        {\bm x(\s,p+1)}_q & = &
        \left\{\begin{array}{lll}
        {\bm x(\s,p)}_q & ~\mbox{if}~ q \in [\nu+p];\\
        \min{(V((\h[\s,\bm x(\s,p);\nu+p])^\circ),\prec_s)} & ~\mbox{if}~ q = \nu+p+1.
        \end{array}\right.
    \end{array}
    $$
    On the other hand, if $V((\h[\s,\bm x(\s,p);\nu+p])^\circ) = \emptyset$, then we define $\bm x(\s,p') = \bm x(\s,p)$ for all $p' > p.$
    We prove that $x(\s,p+1)$ satisfies the property $\mathfrak{P}$ with respect to $\s.$ For this, let $\de,\ep \in [\nu+p+1]$ with $\de < \ep$ and 
    $\s_{\de} = \s_{\ep} = L.$ If $x(\s,p+1) = x(\s,p)$, or 
    $\ep < \nu+p+1$, then we are done. Now, let $x(\s,p+1) \ne x(\s,p)$ and $\ep = \nu+p+1.$ If $x(\s,p+1)_{\ep} = x_{j,f}$, then by using Corollary ~\ref{C-Constructible}, we get $x_j \notin A[\s,\bm x(\s,p);\nu+p].$ Thus, $x(\s,p+1)$ satisfies the property $\mathfrak{P}$ with respect to $\s.$
    
    Let $\gamma(\s)$ be the smallest integer such that $V((\h[\s,\bm x(\s,\gamma(\s));\nu+\gamma(\s)])^\circ) = \emptyset.$ Now, we set
    $$\al(\s) = \nu+\gamma(\s),$$ and $$\bm x(\s) = \bm x(\s,\gamma(\s)).$$
    The sequence $\bm x(\s)$ has $\al(\s)$ terms and it satisfies the property $\mathfrak{P}$ with respect to $\s.$ Note that 
    $\al(\s) \ge \nu+1$ and $\h[\s,\bm x(\s);\nu+1] = (\h(\ell) \setminus (x_{1,1},\dots,x_{1,\nu}))/x_{1,\nu+1}.$ Now using the similar argument given in the proof of Theorem ~\ref{Th-Skeleton}, we observe that the sequence $\bm x(\s)$ satisfies the hypothesis of Lemma ~\ref{L-VDtool}, and hence 
    $\h[\s,\bm x(\s);r]$ is a vertex decomposable hypergraph for all $r \ge \nu+1.$ In particular, $(\h(\ell) \setminus (x_{1,1},\dots,x_{1,\nu}))/x_{1,\nu+1}$ is a vertex decomposable hypergraph.
\end{proof}
Now, we fix some notation for the rest of the section.

\begin{set}\label{SetupMain}
    Let $\Delta =\langle F_1,\dots,F_s\rangle$ be any simplicial complex on the vertex set $V=\{x_1,\dots,x_n\}$ and $W$ be a cycle cover of $\Delta$. Without loss of generality, we assume that $W = \{x_1,\dots,x_m\}.$ Let $$\bar{\Delta} = \Delta(\Gamma(x_1), \dots, \Gamma(x_m))$$ be the simplicial complex obtained from $\Delta$ by attaching non-pure skeleton complex $\Gamma(x_i)$ to $\Delta$ at vertex $x_i$ for all $i\in [m].$ For each $i \in [m]$, let 
    $$\Gamma(x_i) = \Gamma_{i,1}^{(s_{i,1})}\cup \cdots \cup\Gamma_{i,\et_i}^{(s_{i,\et_i})},$$ where $\Gamma_{i,j}$ is a simplex and 
    $s_{i,j} \in [\dim{\Gamma_{i,j}}].$ Since each $\Gamma(x_i)$ is a non-pure skeleton complex, there exists $j \in [\et_i]$ such that 
    $s_{i,j} = \dim{\Gamma_{i,j}}.$ Without loss of generality, we may assume that for each $i \in [m]$, $s_{i,1} = \dim{\Gamma_{i,1}}.$   
    Further, for each $i \in [m]$, let $$\Gamma_{i,1} = \langle F_{s+i} = \{x_i,x_{n_{i-1}+1},\dots,x_{n_i}\} \rangle$$ with $n = n_0 < n_1 < \cdots < n_m.$ Then, we can write $$\bar{\Delta} =\langle F_1,\dots,F_s,F_{s+1},\dots,F_{s+m},\dots,F_t \rangle,$$ where for each $k \ge s+m+1$, $F_k$ is a facet of $\Gamma(x_i)$ for some $i \in [m].$ 
	
    Let $\Upsilon_i =\{k \in [t]: x_i \in F_k\}$ for $i \in [m].$ We say that a tuple $(\ell_1,\dots,\ell_t)\in \mathbb{N}_{>0}^t$ satisfies \emph{condition $(\star)$} if for every $i \in [m]$, $\ell_{s+i} \ge \ell_k$ for all $k \in \Upsilon_i.$ Now, let $\h = \h(\bar{\Delta})$ and 
    $(\ell_1,\dots,\ell_t) \in \mathbb{N}_{>0}^t$ satisfies condition $(\star).$ For each infinite string $\s$ in $\Cla$, we define a finite sequence 
    $\bm x(\s)$ of vertices in $\h(\ell_1,\dots,\ell_t).$ For this, we define a sequence $(\bm x(\s,p) )_{p = 1}^{\infty}$, where $\bm x(\s,p)$ is a finite sequence of vertices in $\h(\ell_1,\dots,\ell_t)$, recursively as follows: For $p = 1$,	define 
    $\bm x(\s,1) = (x_{1,1}).$ Now, let $p \ge 1$ be an integer and sequence $\bm x(\s, p)$ is defined. Consider the set 
    $$V[\s;p] = \{x_{i,c} \in V((\h[\s,\bm x(\s,p);p])^\circ): x_i \in W\}.$$
    If $V[\s;p] \ne \emptyset$, then we define
    $\bm x(\s,p+1)$ as
    $$
    \begin{array}{ccl}
        {\bm x(\s,p+1)}_q & = &
        \left\{\begin{array}{lll}
        {\bm x(\s,p)}_q & ~\mbox{if}~ q \in [p];\\
        \min{(V[\s;p],\prec_v)} & ~\mbox{if}~ q = p+1.
        \end{array}\right.
    \end{array}
    $$
    On the other hand, if $V[\s;p] = \emptyset$, then we define $\bm x(\s,p') = \bm x(\s,p)$ for all $p' > p.$ Let $\al(\s)$ be the smallest integer such that $V[\s;\al(\s)] = \emptyset.$ We define $$\bm x(\s) = \bm x(\s,\al(\s)).$$
    Note that $\bm x(\s,p)$ has $p$ terms for each $p \in [\al(\s)].$ Thus, in particular, $\bm x(\s)$ has $\al(\s)$ terms. For each $r \ge 0$, we write
    $A_r = A[\s,\bm x(\s);r]$ and $B_r = B[\s,\bm x(\s);r].$ Further, if $r \in \mathbb{N}$ and $k \in [t]$, then we set $\ell_{k,r} = \max\{0,\ell_k - d_{k,r}\}$, where
    $$d_{k,r} = -|F_k \cap A_r|+\sum\limits_{\substack{x_{i_p} \in F_k \cap A_r,\\x_{i_p,c_p} = {\bm x}_{p}}} c_p.$$
\end{set}

\begin{rmk}\label{R-MainRemark}
    With the notation of Set-up ~\ref{SetupMain}, 
    \begin{enumerate}[(a)] 
		\item Let $k \in [t].$ Then either $F_k \subset V$, or $F_k \cap V \subset \{x_i\}$ for some $i \in [m].$ Further, if  
        $$F_k \cap \{x_{n_{i-1}+1},\dots,x_{n_i}\} \ne \emptyset$$ for some $i \in [m]$, then $k = s+i.$  
		
		\item If $r \ge \al(\s)$, then $\bm x(\s,r) = \bm x(\s,\al(\s))$ and $\h[\s,\bm x(\s,r);r] = \h[\s,\bm x(\s,\al(\s));\al(\s)].$
        Therefore, we get $V[\s;r] = V[\s;\al(\s)] = \emptyset.$

        \item We show that the sequence $\bm x(\s,p)$ satisfies the property $\mathfrak{P}$ with respect to $\s$ for all $p \ge 1$ by using induction on $p.$ The result holds for $p = 1.$ Now, let $p \ge 1$ and the sequence $\bm x(\s,p)$ satisfies the property $\mathfrak{P}$ with respect to $\s.$ Suppose that 
        $\de,\ep \in [p+1]$ with $\de < \ep$ and $\s_{\de} = \s_{\ep} = L.$ If $\bm x(\s,p+1) = \bm x(\s,p)$, or $\ep < p+1$, then by induction hypothesis, 
        $\bm x(\s,p+1)$ satisfies the property $\mathfrak{P}$ with respect to $\s.$ Now, let $\bm x(\s,p+1) \ne \bm x(\s,p)$ and $\ep = p+1.$ 
        Then $$\bm x(\s,p+1)_{\ep} = \min(V[\s;p],\prec_v).$$ If $\bm x(\s,p+1)_{\ep} = x_{j,f}$, then it follows from Corollary ~\ref{C-Constructible} that 
        $x_j \notin A[\s,\bm x(\s,p);p].$ Therefore, $\bm x(\s,p+1)$ satisfies the property $\mathfrak{P}$ with respect to $\s.$ In particular, $\bm x(\s)$ satisfies the property $\mathfrak{P}$ with respect to $\s.$

        \item Let $0 \le p' < \al(\s)$ and $\s' \in \Cla(\s,p').$ If $\al(\s') \le p'$, then 
        $\bm x(\s',\al(\s')) = \bm x(\s,\al(\s'))$, and hence 
        $$\h[\s',\bm x(\s',\al(\s'));\al(\s')] = \h[\s,\bm x(\s,\al(\s'));\al(\s')].$$ 
        Now, the fact $V[\s';\al(\s')] = \emptyset$ implies that $\al(\s') \ge \al(\s) > p'$, a contradiction. 
        Thus $p' <  \al(\s')$ for all $\s' \in \Cla(\s,p').$ Also, we have ${\bm x(\s)}_r = {\bm x(\s')}_r$ for all $r \in [p'+1].$
        
        \item Let $\bm x(\s) = (x_{i_q,c_q})_{q = 1}^{\al(\s)}.$ Then $c_q = \min\{c:x_{i_q,c} \in V((\h[\s,\bm x(\s);q-1])^\circ)\}$ and 
        $i_q = \min\{i: x_i \in W, x_{i,c} \in V((\h[\s,\bm x(\s);q-1])^\circ) ~\mbox{for some}~ c \ge 1\}.$ Further, $x_{i_q} \in W$ for all $q \in [\al(\s)]$, and hence 
        $A_q, B_q \subset W$ for all $q \ge 1.$ 
    \end{enumerate}
\end{rmk}

\begin{lem}\label{L-CHSX}
    With the notation of Set-up ~\ref{SetupMain}, let $\bm x = (x_{i_q,c_q})_{q = 1}^{\al(\s)}$ and $I = \{i_r \in [m]: r \in [\al(\s)]\}.$ Further, for $i \in I$, let $\mu_i = \min{C_i}$ and $\nu_i+1 = \max{C_i}$, where
    $$C_i = \{c:x_{i,c} = x_{i_r,c_r} ~\mbox{for some}~ r \in [\al]\}.$$
    \begin{enumerate}
        \item For each $\mu_i \le c \le \nu_i+1$, $c \in C_i.$

        \item If $x_{i,\mu_i} = x_{i_r,c_r}$ and $\mu_i < \nu_i+1$, then $\s_q = D$ for all $r \le q \le r+\nu_i-\mu_i.$ 
    \end{enumerate}
\end{lem}

\begin{proof} 
    (1) If $\mu_i = \nu_i+1$, then we are done. Therefore, we suppose that $\mu_i < \nu_i+1.$
    Note that $\mu_i \in C_i.$ Suppose $\mu_i < c \le \nu_i+1$ and $c-1 \in C_i.$ Choose $r \in [\al]$ such that $x_{i,c-1} = x_{i_r,c_r}.$
    It is suffices to prove that $x_{i,c} \in V((\h[\s,\bm x(\s);r])^\circ).$ For $c' \in C_i$ with $\mu_i < c'$, let $x_{i,c'} = x_{i_{r'},c_{r'}}.$
    Then we observe that $r < r'.$ Since 
    $$V((\h[\s,\bm x(\s);r+\nu_i-\mu_i])^\circ) \subset V((\h[\s,\bm x(\s);r])^\circ),$$
    we get $x_{i,\nu_i+1} \in V((\h[\s,\bm x(\s);r])^\circ).$
    Let $E = \{x_{j_1,f_1},\dots,x_{j_a,f_a}\}\in \e((\h[\s,\bm x(\s);r])^\circ)$ be such that $x_{i,\nu_i+1} \in E.$ Using Theorem ~\ref{Th-Constructible}, there exists 
    $k \in [t]$ such that $\{x_{j_1},\dots,x_{j_a}\} = F_k \setminus A_r$ and $\mathbf{f} = (f_1,\dots,f_a) \in [\ell_k]^a$ with 
    $$f_q > c_p ~\mbox{if}~ x_{j_q} = x_{i_p} \in B_r \setminus A_r ~\mbox{and}~ |\mathbf{f}| \le \ell_{k,r}+a-1.$$ Assume that $x_{i,\nu_i+1} = x_{j_1,f_1}.$ Then $i = j_1$ and $\nu_i+1 = f_1.$ Consider the set $$F = \{x_{j_1,c},x_{j_2,f_2},\dots,x_{j_a,f_a}\}.$$
    Suppose that $x_{j_1} = x_{i_p} \in B_r \setminus A_r.$ If $p = r$, then $c_p = c_r = c-1$, which implies $c > c_p.$ On the other hand if $p < r$, then $x_{j_1} = x_{i_p} \in B_{r-1} \setminus A_{r-1}.$ Since $x_{i_r,c_r} \in V((\h[\s,\bm x(\s);r])^\circ)$, it follows from Corollary ~\ref{C-Constructible} that 
    $c > c-1 > c_p.$ Now, the fact $|\mathbf{f}|-f_1+c \le |\mathbf{f}| \le \ell_{k,r}+a-1$ implies that $F$ is constructible in $\h[\s,\bm x(\s);r].$ Again, using Theorem ~\ref{Th-Constructible}, there exist an edge $F' \in \e(\h[\s,\bm x(\s);r])$ such that $F' \subset F.$ Since $F',E \in \e(\h[\s,\bm x(\s);r])$, we get 
    $F' \not \subset E$, and hence $x_{i,c} \in F'.$ Further, if $F' = \{x_{i,c}\}$, then by using Theorem ~\ref{Th-Constructible}, there exists $k' \in [t]$ such that 
    $\{x_i\} = F_{k'} \setminus A_r$ and $c \in [\ell_{k',r}].$  Moreover, since $x_{i,\nu_i+1} \in V((\h[\s,\bm x(\s);r])^\circ)$, in view of Corollary ~\ref{C-Constructible},
    we see that $x_i \notin A_r.$ Thus, $A_r = A_{r-1}$, and hence $\ell_{k',r} = \ell_{k',r-1}.$ This implies that $\{x_{i,c-1}\}$ is constructible in $\h[\s,\bm x(\s);r-1].$ Therefore, by Theorem ~\ref{Th-Constructible}, we obtain
    $\{x_{i,c-1}\} \in \e(\h[\s,\bm x(\s);r-1]).$ This implies that $x_{i,c-1} \notin V((\h[\s,\bm x(\s);r-1])^\circ)$, a contradiction. Therefore, we get
    $\{x_{i,c}\} \subsetneq F'$ which implies that $x_{i,c} \in V((\h[\s,\bm x(\s);r])^\circ).$ 
    
    (2) Let $r \le q \le r+\nu_i-\mu_i.$ Using part (1), for each $r \le p \le r+\nu_i-\mu_i+1$, we have $i_p = i.$ Now, since 
    $x_{i_{q+1},c_{q+1}} \in V((\h[\s,\bm x(\s);q])^\circ)$, it follows from Corollary ~\ref{C-Constructible} that $x_{i_{q+1}} \notin A_q$, that is, $x_q \notin A_q.$ This implies that $\s_q = D.$
     
\end{proof}

\begin{lem}\label{L-ImConstructible}
    With the notation of Set-up ~\ref{SetupMain},
    let $r \ge \al(\s)$ and $E  = \{x_{j_1,f_1},\dots,x_{j_a,f_a}\}$ is an edge in $(\h[\s,\bm x(\s);r])^\circ.$ Then there exists $k \in [t]$ such that 
    $\{x_{j_1},\dots,x_{j_a}\} = F_k \setminus W$ and 
    $\mathbf{f} =(f_1,\dots,f_a) \in [\ell_k]^a$ with 
    $|\mathbf{f}| \le \ell_{k,r}+a-1.$
\end{lem}

\begin{proof}
    Let $\bm x(\s) = (x_{i_q,c_q})_{q = 1}^{\al(\s)}.$ 
    By applying Theorem ~\ref{Th-Constructible}, there exists $k \in [t]$ such that $\{x_{j_1},\dots,x_{j_a}\} = F_k \setminus A_r$ and 
    $\mathbf{f} = (f_1,\dots,f_a) \in [\ell_k]^a$ with 
    $$f_q > c_p ~\mbox{if}~ x_{j_q} = x_{i_p} \in B_r \setminus A_r ~\mbox{and}~ |\mathbf{f}| \le \ell_{k,r}+a-1.$$ 
    Further, by Remark ~\ref{R-MainRemark}(b), we have $V[\s;r] = \emptyset.$ Since $x_{j_q,f_q} \in V((\h[\s,\bm x(\s,r);r])^\circ)$ for all $q \in [a]$, we get 
    $x_{j_q} \notin W$ for all $q \in [a].$ Now, Remark ~\ref{R-MainRemark}(e) implies that $F_k \setminus W \subset F_k \setminus A_r$, and hence 
    $F_k \setminus W = F_k \setminus A_r$ and $(F_k \setminus A_r) \cap B_r = \emptyset.$ This proves the lemma.
\end{proof}

\begin{lem}\label{L-Edge}
    With notation as in Set-up ~\ref{SetupMain}, let $\bm x(\s) = (x_{i_q,c_q})_{q = 1}^{\al(\s)}$, $r \in [\al(\s)]$, and $i = i_r.$ Then
    $$E = \{x_{i,c_r},x_{n_{i-1}+1,1},\dots,x_{n_i,1}\} \in \e((\h[\s,\bm x(\s);r-1])^\circ).$$
\end{lem}	

\begin{proof}
    Since $(\ell_1,\dots,\ell_t) \in \mathbb{N}^t_{>0}$ satisfies condition $(\star)$, it follows that 
    $$\ell_{s+i} = \max\{c:x_{i,c}\in V(\h(\ell_1,\dots,\ell_t))\},$$
    and hence $c_r \le \ell_{s+i}.$ Further, the fact $x_{i,c_r} \in V((\h[\s,\bm x(\s);r-1])^\circ)$ and Corollary ~\ref{C-Constructible} implies that $x_i \notin A_{r-1}.$ Note that $\{x_{n_{i-1}+1},\dots,x_{n_i}\} \cap W = \emptyset.$ Therefore, by applying Remark ~\ref{R-MainRemark}(e), we get $F_{s+i}\cap A_{r-1} = \emptyset$ and 
    $F_{s+i}\cap B_{r-1} \subset \{x_i\}$ (Recall that $F_{s+i} = \{x_i,x_{n_{i-1}+1},\dots,x_{n_i}\}$). Thus, $F_{s+i} = F_{s+i} \setminus A_{r-1}$ and $\ell_{s+i,r-1} = \ell_{s+i}.$ Also if $a = n_i-n_{i-1}+1$, then $c_r+a-1 \le \ell_{s+i,r-1}+a-1.$ Note that $x_j \notin B_{r-1} \setminus A_{r-1}$ for all 
    $j \in \{n_{i-1}+1,\dots,n_i\}.$ Thus, it follows from Remark ~\ref{R-BasicRemark}(a) that $E$ is constructible in $\h[\s,\bm x(\s);r-1].$ Using Theorem ~\ref{Th-Constructible}, there exists an edge $E' = \{x_{j_1,f_1},\dots,x_{j_a,f_a}\} \in \e(\h[\s,\bm x(\s);r-1])$ with $E' \subset E$, where 
    $$\{x_{j_1},\dots,x_{j_a}\} = F_k \setminus A_{r-1}$$ 
    and $\mathbf{f} = (f_1,\dots,f_a) \in [\ell_k]^a$ with $|\mathbf{f}|\le \ell_{k,r-1}+a-1$ for some $k \in [t].$ 
    To obtain the statement, it suffices to prove that $k = s+i.$ On the contrary, suppose $k \ne s+i.$ Then, by Remark ~\ref{R-MainRemark}(a), we have
    $$(F_k \setminus A_{r-1}) \cap \{x_{n_{i-1}+1},\dots,x_{n_i}\} = \emptyset.$$ 
    The fact $E' \subset E$ implies that $E' = \{x_{i,c_r}\}.$ This is a contradiction to the Remark ~\ref{R-BasicRemark}(b). 
\end{proof}

\begin{lem}\label{L-SheddingV}
    With notation as in Set-up ~\ref{SetupMain}, let $\bm x(\s) = (x_{i_q,c_q})_{q = 1}^{\al(\s)}$ and $r \in [\al(\s)].$ Then $x_{i_r,c_r}$ is a shedding vertex of hypergraph $(\h[\s,\bm x(\s);r-1])^\circ.$
\end{lem}

\begin{proof}
    Let $i = i_r$ and $U \in \link_{\mathcal{D}}(x_{i_r,c_r})$, where $\mathcal{D} = \Delta((\h[\s,\bm x(\s);r-1])^\circ).$ Since 
    $U\cup \{x_{i_r,c_r}\} \in \mathcal{D}$, it follows from Lemma ~\ref{L-Edge} that there exists $j \in \{n_{i-1}+1,\dots,n_i\}$ such that 
    $$\{x_{j,1},\dots,x_{j,{\ell}_{s+i}}\}\cap V(\h[\s,\bm x(\s);r-1])^\circ) \not \subset U.$$ 
    Let $f$ be the largest positive integer satisfying $x_{j,f}\in V(\h[\s,\bm x(\s);r-1])^\circ) \setminus U.$ We claim that $U \cup\{x_{j,f}\}\in \mathcal{D}.$ Suppose that there exists an edge $E = \{x_{j_1,f_1},\dots,x_{j_a,f_a}\} \in \e((\h[\s,\bm x(\s);r-1])^\circ)$ with $E \subset U \cup \{x_{j,f}\}.$ 
    By Theorem ~\ref{Th-Constructible}, there exists an $k \in [t]$ such that $\{x_{j_1},\dots,x_{j_a}\} = F_k \setminus A_{r-1}$ and 
    $\mathbf{f} = (f_1,\dots,f_a) \in [\ell_k]^a$ with
    $$f_q > c_p ~\mbox{if}~ x_{j_q} = x_{i_p} \in B_{r-1} \setminus A_{r-1} ~\mbox{and}~ |\mathbf{f}|\le \ell_{k,r-1}+a-1.$$
    Since $U \in \mathcal{D}$, it follows that $x_{j,f} \in E.$ This implies that $(F_k \setminus A_{r-1})\cap \{x_{n_{i-1}+1},\dots,x_{n_i}\} \ne \emptyset.$ 
    Using Remark ~\ref{R-MainRemark}(a), we get $k = s+i.$ Also, by Corollary ~\ref{C-Constructible}, we have $x_i \notin A_{r-1}.$
    Since $$\{x_{n_{i-1}+1},\dots,x_{n_i}\} \cap W = \emptyset,$$ it follows from Remark ~\ref{R-MainRemark}(e) that $F_k \setminus A_{r-1} = F_{s+i}$ and 
    $F_{s+i}\cap B_{r-1} \subset \{x_i\}.$ This implies that $x_i \in F_k \setminus A_{r-1}.$ Without loss of generality, assume that 
    $x_j = x_{j_1}$ and $x_i = x_{j_2}.$ Then $f = f_1.$ Since $x_{j_2,f_2} \in E \setminus \{x_{j,f}\} \subset U$ and $U \in \link_{\mathcal{D}}(x_{i_r,c_r})$, it follows from  Remark ~\ref{R-MainRemark}(e) that $f_2 > c_r \ge 1.$
    Again, since $|\mathbf{f}|\le \ell_{k,r-1}+a-1$ and $\ell_{k,r-1} \le \ell_k = \ell_{s+i}$, it follows that $f < \ell_{s+i}$, and hence the set 
    $$F = \{x_{j_1,f_1+1},x_{j_2,c_r},\dots,x_{j_a,f_a}\} \subset V(\h(\ell_1,\dots,\ell_t)).$$
    Note that $(f_1+1)+c_r+\cdots +f_a \le \ell_{k,r-1}+a-1.$ Thus, by the fact $F_{s+i}\cap B_{r-1} \subset \{x_i\}$ and Remark ~\ref{R-BasicRemark}(a), $F$ is constructible in $\h[\s,\bm x(\s);r-1].$ By Theorem ~\ref{Th-Constructible}, there exists an edge 
    $F' = \{x_{j'_1,f'_1},\dots,x_{j'_b,f'_b}\} \in \e(\h[\s,\bm x(\s);r-1])$ such that $F' \subset F$, where 
    $\{x_{j'_1},\dots,x_{j'_b}\} = F_{k'}\setminus A_{r-1}$ and 
    $$\mathbf{f}' = (f'_1,\dots,f'_b) \in [\ell_{k'}]^b ~\mbox{with}~ |\mathbf{f}'| \le \ell_{k',r-1}+b-1$$ for some $k' \in [t].$
    We claim that $k' = s+i.$ On the contrary, suppose that $k' \ne s+i.$
    Then, by Remark ~\ref{R-MainRemark}(a), we have 
    $$(F_{k'}\setminus A_{r-1})\cap \{x_{n_{i-1}+1},\dots,x_{n_i}\} = \emptyset.$$  Since $F_{k'} \setminus A_{r-1} \subset F_{s+i}$, it follows that
    $F_{k'} \setminus A_{r-1} = \{x_{j_2}\}.$ Therefore $F' = \{x_{i,c_r}\}.$ This is a contradiction to the Remark ~\ref{R-BasicRemark}(b).
    Thus, we obtain $k' = s+i$, i.e. $F' = F.$ Now, by maximality of $f$, we have $F \subset U \cup \{x_{i_r,c_r}\}.$ This contradicts the fact that 
    $U \notin \link_{\mathcal{D}}(x_{i_r,c_r}).$ Since $x_j \ne x_i$, it follows that $U \cup \{x_{j,f}\} \in \del_{\mathcal{D}}(x_{i_r,c_r}).$ This proves the lemma.
\end{proof} 

\begin{rmk}\label{R-HelpRemark}
    With notation as in Set-up ~\ref{SetupMain}, let $E = \{x_{j_1,f_1},\dots,x_{j_a,f_a}\} \subset V(\h(\ell_1,\dots,\ell_t)).$ Then we write $\tilde{E}$ for the set $\{x_{j_1},\dots,x_{j_a}\} \subset V(\h).$ Recall that $\Delta$ is a simplicial complex on the vertex set $V = \{x_1,\dots,x_n\}$ and 
    $W = \{x_1,\dots,x_m\}$ is a cycle cover of $\Delta.$
    For simplicity, let $\al = \al(\s).$ Consider the set $$\mathcal{X} = \{\tilde{E}:E \in \e((\h[\s,\bm x(\s);\al])^\circ)\}.$$ 
    Let $\mathcal{X}_{\max}$ denote the set of all maximal elements of $\mathcal{X}$ with respect to inclusion. 
    \begin{enumerate}[(a)]
        \item First assume that $\{X \in \mathcal{X}\setminus \mathcal{X}_{\max}:X\subset V\} \ne \emptyset$ and 
        $\{X_1,\dots,X_{\rho}\} = \{X \in \mathcal{X}\setminus \mathcal{X}_{\max}:X\subset V\}.$ Let $Y = \{y_1,\dots,y_{\rho}\}$ and 
        $Z = Z.$ Consider the simplicial complex $\mathcal{T}$ on the vertex set $V' = (V \setminus W) \cup Y \cup Z$ 
        given by $$\mathcal{T} = \langle X:X \in \mathcal{X}_{\max}, X\subset V \rangle \cup 
        \langle X_l \cup\{y_l\}:l \in [\rho] \rangle \cup \langle \{y_l,z_l\}:l \in [\rho] \rangle.$$ 
        Let $T_1,\dots,T_{\tau}$ be all the facets of $\mathcal{T}$ such that $T_h = \{y_h,z_h\}$ for all $h \in [\rho].$ Then, for each $h \in [\rho]$, $T_h$ is a good leaf of $\mathcal{T}.$ 
	    \begin{enumerate}[(i)]
            \item We prove that $\mathcal{T}$ is a forest. Suppose that $\mathcal{T}$ has a special cycle $$w_1,T_{h_1},w_2,\dots,w_q,T_{h_q},w_{q+1} = w_1$$ with 
            $q \ge 3.$ Since $T_h$ is a good leaf of $\mathcal{T}$ for every $h \in [\rho]$, it follows that $h_p \ge \rho+1$ for all $p \in [q].$
            Thus, $\{w_1,\dots,w_q\}\cap Z = \emptyset.$ We claim that $\{w_1,\dots,w_q\}\cap Y = \emptyset.$ On the contrary, assume that 
            $\{w_1,\dots,w_q\}\cap Y \ne  \emptyset.$ Without loss of generality, we further assume that $y_1 \in \{w_1,\dots,w_q\}.$ Then 
            $y_1 \in T_{h_p} \cap T_{h_{p'}}$ for some $p \ne p'.$ Thus, we must have $T_{h_p} = X_l \cup \{y_l\}$ and 
            $T_{h_{p'}} = X_{l'} \cup \{y_{l'}\}$ for some $l,l' \in [\rho]$ with $l \ne l'.$ This implies that either $y_1 \in X_l$ or $y_1 \in X_{l'}$, a contradiction. Thus, we obtain $$\{w_1,\dots,w_q\}\cap (Y \cup Z) = \emptyset.$$ 
            Let $p \in [q]$. Then either $T_{h_p} \in \mathcal{X}_{\max}$ or $T_{h_p} = X_l \cup\{y_l\}$ for some $l \in [\rho].$ Therefore, 
            $T_{h_p} \setminus Y = \tilde{E}$ for some $E \in \e((\h[\s,\bm x(\s);\al])^\circ).$ Using Lemma ~\ref{L-ImConstructible}, we get 
            $$T_{h_p} \setminus Y = F_{k_p}\setminus W$$ for some $k_p \in [s].$ Now, $w_1,F_{k_1},w_2,\dots,w_q,F_{k_q},w_{q+1} = w_1$ is a special cycle in $\Delta$ and $\{w_1,\dots,w_q\}\cap W = \emptyset.$ This contradicts the fact that $W$ is a cycle cover of $\Delta.$ Thus, $\mathcal{T}$ has no special cycle of length $\ge 3$, and hence it follows from Theorem ~\ref{Th-Forest} that $\mathcal{T}$ is a forest, and hence $\mathcal{T}$ has a good leaf order. Without loss of generality, we may assume that $T_h$ is a good leaf of subcollection $\langle T_h,\dots,T_{\tau} \rangle$ of $\mathcal{T}$ for all $h \ge \rho+1.$
            
            \item For $h \in [\rho]$, let $\ell'_h = 1.$ Now, suppose $h \ge \rho+1.$ Then $T_h \setminus Y = \tilde{E}$ for some 
            $E \in \e((\h[\s,\bm x(\s);\al])^\circ).$ By Lemma ~\ref{L-ImConstructible}, $T_h \setminus Y = F_k \setminus W$ for some $k \in [s].$
            Thus the set $$U_h = \{\ell_{k,\al}:k \in [s],T_h \setminus Y = F_k \setminus W\}$$ is non-empty. Choose $k_h \in [s]$ such that 
            $\ell_{k_h,\al} = \max{U_h}.$
            Further, suppose $$T_h \setminus Y = \{x_{j_1},\dots,x_{j_a}\}.$$ 
            For $p \in [a]$, we set $$c'_p =\min\{c:x_{j_p,c} \in V((\h[\s,\bm x(\s);\al])^\circ)\}-1.$$ Since $x_{j_p,c} \in E$ for some $c \ge 1$, $c'_p$ is well defined. Now, let $\ell'_h = \max\{0,\ell_{k_h,\al}-d_h\}$, where $d_h =\sum_{p \in [a]} c'_p.$
            Now, if $\mathbf{f} = (f_1,\dots,f_a) \in [\ell_{k_h}]^a$ and $\mathbf{f}' = (f_1-c'_1,\dots,f_a-c'_a) \in [\ell'_h]^a$, then 
            $$|\mathbf{f}| \le \ell_{k_h,\al}+a-1 ~\mbox{if and only if}~ |\mathbf{f}'| \le \ell'_h+a-1.$$
	        
            \item Let $\h' = \h(\mathcal{T}).$ Then $\e(\h') = \{T_1,\dots,T_{\tau}\}.$ Further, let $\p$ be an infinite string in $\Cla$ with $\p_r = L$ for all $r \in [\rho]$ and $\bm y = (y_{1,1},\dots,y_{\rho,1})$ be a finite sequence of vertices in $V(\h'(\ell'_1,\dots,\ell'_{\tau})).$ 
            Then, we have $$\h'[\p,\bm y;\rho] = \h'(\ell'_1,\dots,\ell'_{\tau})/\{y_{1,1},\dots,y_{\rho,1}\},$$ where $\h'[\p,\bm y;\rho]$ is the $\rho$th hypergraph determined by string $\p$ and sequence $\bm y.$ Also, we have $A[\p,\bm y;\rho] = Y$ and $B[\p,\bm y;\rho] = \emptyset$, where
            $A[\p,\bm y;\rho]$ is the $\rho$th set of vertices of contraction determined by string $\p$ and sequence $\bm y$, and $B[\p,\bm y;\rho]$ is the $\rho$th set of vertices of deletion determined by string $\p$ and sequence $\bm y.$ Let $F = \{x_{j_1,f'_1},\dots,x_{j_a,f'_a}\}$, where 
            $\{x_{j_1},\dots,x_{j_a}\} \subset V.$ Then $F$ is constructible in $\h'[\p,\bm y;\rho]$ if and only if there exists $h \ge \rho+1$ such that 
            $$\{x_{j_1},\dots,x_{j_a}\} = T_h \setminus Y$$ and 
            $\mathbf{f}' = (f'_1,\dots,f'_a) \in [\ell'_h]^a$ with $|\mathbf{f}'| \le \ell'_h+a-1.$
	    \end{enumerate}
	    
	    \item Suppose that $\{X \in \mathcal{X}\setminus \mathcal{X}_{\max}:X\subset V\} = \emptyset.$
        Let $\mathcal{K} = \langle X : X \in \mathcal{X}, X\subset V \rangle.$ 
        \begin{enumerate}[(i)]
            \item Let $K_1,\dots,K_{\sigma}$ be all the facets of $\mathcal{K}.$ By following the similar procedure as in $(a)$, we can say that $\mathcal{K}$ is a forest. 
            
            \item Let $h \in [\sigma].$ Then $K_h = \tilde{E}$ for some $E \in \e((\h[\s,\bm x(\s);\al])^\circ).$ By Lemma ~\ref{L-ImConstructible}, there exists 
            $k \in [s]$ such that $K_h = F_k \setminus W.$ Thus the set $$U'_h = \{\ell_{k,\al}:k \in [s], K_h = F_k \setminus W\}$$ is non-empty. 
            Choose $k_h \in [s]$ such that $\ell_{k_h,\al} = \max{U'_h}.$ Now, suppose that $$K_h = \{x_{j_1}\dots x_{j_a}\}.$$ For $p \in [a]$, we set 
            $$c'_p =\min\{c:x_{j_p,c} \in V((\h[\s,\bm x(\s);\al])^\circ)\}-1.$$
            As in part (a), let $d_h = \sum_{p \in [a]} c'_p$ and $\ell'_h =\max\{0,\ell_{k_h,\al}-d_h\}$. Then
            $$|\mathbf{f}| \le \ell_{k_h,\al}+a-1 ~\mbox{if and only if}~ |\mathbf{f}'| \le \ell'_h+a-1,$$ where 
            $\mathbf{f} = (f_1,\dots,f_a) \in [\ell_{k_h}]$ and $\mathbf{f}' = (f_1-c'_1,\dots,f_a-c'_a)\in [\ell'_h]^a.$
        \end{enumerate}	     
    \end{enumerate}    
\end{rmk}

The following lemmas are useful in proving Theorem ~\ref{Th-MainResult}.

\begin{lem}\label{L-Tool}
    With notation as in Set-up ~\ref{SetupMain} and Remark ~\ref{R-HelpRemark}(a), let $E = \{x_{j_1,f_1},\dots,x_{j_a,f_a}\}$ be an edge in $(\h[\s,\bm x(\s);\al])^\circ$ with $\tilde{E}\subset V .$ Then $F = \{x_{j_1,f'_1},\dots,x_{j_a,f'_a}\}$ is an edge in $\h'[\p,\bm y;\rho]$, where $f'_p = f_p-c'_p$ for all $p \in [a].$
\end{lem}

\begin{proof}
    Let $E = \{x_{j_1,f_1},\dots,x_{j_a,f_a}\} \in \e((\h[\s,\bm x(\s);\al])^\circ)$ with $\tilde{E} \subset V.$ By Lemma ~\ref{L-ImConstructible}, there exists $k \in [s]$ such that $\tilde{E} = F_k \setminus W$ and $\mathbf{f} = (f_1,\dots,f_a) \in [\ell_k]^a$ with $|\mathbf{f}| \le \ell_{k,\al}+a-1.$
    Now, since $\tilde{E} = F_k \setminus W \in \mathcal{X}$ and $\tilde{E}\subset V$, it follows that $F_k \setminus W = T_h\setminus Y$ for some $h \ge \rho+1.$
    Note that $\ell_{k,\al} \in U_h$, and hence $|\mathbf{f}| \le \ell_{k_h,\al}+a-1.$ It follows from part (ii) of Remark ~\ref{R-HelpRemark}(a) that
    $|\mathbf{f}'| \le \ell'_h+a-1$, where $\mathbf{f}' = (f'_1,\dots,f'_a).$ Thus, using part (iii) of Remark ~\ref{R-HelpRemark}(a), we get 
    $F = \{x_{j_1,f'_1},\dots,x_{j_a,f'_a}\}$ is constructible in $\h'[\p,\bm y;\rho].$ By Theorem ~\ref{Th-Constructible}, there exists a edge 
    $F' \in \e(\h'[\p,\bm y;\rho])$ with $F' \subset F.$ We may assume without loss of generality that 
    $F' = \{x_{j_1,f'_1},\dots,x_{j_b,f'_b}\}$ for some $b \in [a].$ Since $F'$ is constructible $\h'[\p,\bm y;\rho]$ and 
    $\tilde{E}\subset V$, there exists $h' \ge \rho+1$ such that $\{x_{j_1},\dots,x_{j_b}\} = T_{h'} \setminus Y$ and 
    $$\mathbf{e'} =(f'_1,\dots,f'_b) \in [\ell'_{h'}]^b ~\mbox{with}~ |\mathbf{e'}| \le \ell'_{h'}+b-1.$$ Again, since 
    $T_{h'} \setminus Y \in \mathcal{X}$, there exists an edge $E_1 \in \e((\h[\s,\bm x(\s);\al])^\circ)$ such that 
    $T_{h'} \setminus Y = \tilde{E_1}.$ This implies that $b >1.$ By using part (ii) of Remark ~\ref{R-HelpRemark}(a), we get
    $T_{h'} \setminus Y = F_{k_{h'}} \setminus W$ and $|\mathbf{e}| \le \ell_{k_{h'},\al}+b-1$, where $\mathbf{e} = (f_1,\dots,f_b).$ Observe that 
    $F_{k_{h'}} \setminus A_{\al} \subset F_{k_{h'}} \setminus W.$ Thus, using Remark ~\ref{R-MainRemark}(e), we obtain that $E' = \{x_{j_1,f_1},\dots,x_{j_b,f_b}\}$ is constructible in $\h[\s,\bm x(\s);\al]$, and hence using Theorem ~\ref{Th-Constructible}, there exists an edge $E''$ is 
    $\h[\s,\bm x(\s);\al]$ with $E'' \subset E'\subset E.$ Since $E$ is an edge in $\h[\s,\bm x(\s);\al]$, we must have $F' = F.$
\end{proof}

The following lemma says that the converse of the above lemma holds.

\begin{lem}\label{L-TooLInverse}
    With notation of Set-up ~\ref{SetupMain} and Remark ~\ref{R-HelpRemark}(a), let $F = \{x_{j_1,f'_1},\dots,x_{j_a,f'_a}\}$ be an edge in 
    $(\h'[\p,\bm y;\rho])^\circ.$ Then 
    $E = \{x_{j_1,f_1},\dots,x_{j_a,f_a}\}$ is an edge in  $(\h[\s,\bm x(\s);\al])^\circ$, where $f_p = f'_p+c'_p$ for all $p \in [a].$ 
\end{lem}

\begin{proof}
    Let $F = \{x_{j_1,f'_1},\dots,x_{j_a,f'_a}\} \in \e((\h'[\p,\bm y;\rho])^\circ).$ 
    As seen in the proof of Lemma ~\ref{L-Tool}, 
    $E = \{x_{j_1,f_1},\dots,x_{j_a,f_a}\}$ is constructible in $\h[\s,\bm x(\s);\al].$ Thus, by Theorem ~\ref{Th-Constructible}, there exists a edge 
    $E'$ in $\h[\s,\bm x(\s);\al]$ such that $E' \subset E.$ Assume that $E' = \{x_{j_1,f_1},\dots,x_{j_b,f_b}\}.$ 
    Suppose that $b = 1.$ Since $f_1 \ge c'_p+1$, by using the fact $x_{j_1,c'_p+1} \in V((\h[\s,\bm x(\s);\al])^\circ)$ and Corollary ~\ref{C-Constructible}, we obtain that $\{x_{j_1,c'_p+1}\}$ is constructible in $\h[\s,\bm x(\s);\al].$ Now, it follows from Theorem ~\ref{Th-Constructible} that
    $\{x_{j_1,c'_p+1}\} \in \e(\h[\s,\bm x(\s);\al]).$ Thus, $x_{j_1,c'_p+1} \notin V((\h[\s,\bm x(\s);\al])^\circ)$, a contradiction. Thus, we must have $b > 1.$ Also, we have $\tilde{E'} \subset \tilde{E}.$ By Lemma ~\ref{L-Tool}, $F' = \{x_{j_1,f'_1},\dots,x_{j_b,f'_b}\}$ is an edge in 
    $(\h'[\p,\bm y;\rho])^\circ.$ Since $F \in \e(\h'[\p,\bm y;\rho])^\circ)$ and $F' \subset F$, it follows that $b = a.$ This proves that 
    $E \in \e((\h[\s,\bm x(\s);\al])^\circ).$
\end{proof}

\begin{lem}\label{L-2ndTool}
     With notation of Set-up ~\ref{SetupMain} and Remark ~\ref{R-HelpRemark}(b), let $\h'' = \h(\mathcal{K}).$ Then $E = \{x_{j_1,f_1},\dots,x_{j_a,f_a}\}$ is an edge in  $(\h[\s,\bm x(\s);\al])^\circ$ with $\tilde{E} \subset V$ if and only if $F = \{x_{j_1,f'_1},\dots,x_{j_a,f'_a}\}$ is an edge in $\h''(\ell'_1,\dots,\ell'_{\sigma})$, where 
     $f_p = f'_p+c'_p$ for all $p \in [a].$ 
\end{lem}

\begin{proof}
    We have $\e(\h'') = \{K_1,\dots,K_{\sigma}\}.$ Let $E = \{x_{j_1,f_1},\dots,x_{j_a,f_a}\} \in \e((\h[\s,\bm x(\s);\al])^\circ)$ with $\tilde{E} \subset V.$ Then, it follows from Lemma ~\ref{L-ImConstructible} that there exists 
    $k \in [s]$ such that $\tilde{E} = F_k \setminus W$ and $\mathbf{f} = (f_1,\dots,f_a) \in [\ell_{k}]$ with 
    $|\mathbf{f}| \le \ell_{k,\al}+a-1.$ Also, there exists $h \in [\sigma]$ such that $\tilde{E} = K_h.$ Observe that $\ell_{k,\al} \in U'_h$, and hence 
    $|\mathbf{f}'| \le \ell'_h+a-1$, where $\mathbf{f}' = (f_1-c'_1,\dots,f_a-c'_a)\in [\ell'_h]^a.$ Therefore $F = \{x_{j_1,f'_1},\dots,x_{j_a,f'_a}\}$ is an edge in $\h''(\ell'_1,\dots,\ell'_{\sigma}).$ 
    
    Conversely, let $F = \{x_{j_1,f'_1},\dots,x_{j_a,f'_a}\}$ is an edge in $\h''(\ell'_1,\dots,\ell'_{\sigma}).$ By Construction ~\ref{Construction}, there exists 
    $h \in [\sigma]$ such that $\{x_{j_1},\dots,x_{j_a}\} = K_h$ and $\mathbf{f}' = (f_1-c'_1,\dots,f_a-c'_a)\in [\ell'_h]^a$ with $|\mathbf{f}'| \le \ell'_h+a-1.$
    Now, using part(ii) of Remark ~\ref{R-HelpRemark}(b), we get $K_h = F_{k_h} \setminus W = F_{k_h} \setminus A_{\al}$ and $|\mathbf{f}| \le \ell_{k_h,\al}+a-1.$ Also, note that 
    $(F_{k_h} \setminus A_{\al}) \cap B_{\al} = \emptyset.$ Therefore, $E = \{x_{j_1,f_1},\dots,x_{j_a,f_a}\}$ is constructible in 
    $\mathcal{H}[\mathcal{S},\bm x(\s);\al].$ Thus, again by Theorem ~\ref{Th-Constructible}, there exists an edge $E'$ in 
    $\mathcal{H}[\mathcal{S},\bm x(\s);\al]$ such that $E' \subset E.$ Let $E' = \{x_{j_1,f_1},\dots,x_{j_b,f_b}\}.$ Suppose that $b = 1.$ Since $f_1 \ge c'_p+1$, by using the fact $x_{j_1,c'_p+1} \in V((\h[\s,\bm x(\s);\al])^\circ)$ and Corollary ~\ref{C-Constructible}, we obtain that $\{x_{j_1,c'_p+1}\}$ is constructible in $\h[\s,\bm x(\s);\al].$ Now, it follows from Theorem ~\ref{Th-Constructible} that $\{x_{j_1,c'_p+1}\} \in \e(\h[\s,\bm x(\s);\al]).$ Therefore we get $x_{j_1,c'_p+1} \notin V((\h[\s,\bm x(\s);\al])^\circ)$, a contradiction. Thus, $b > 1.$ Since $\tilde{E'} \subset \tilde{E}$, as seen above 
    $\{x_{j_1,f'_1},\dots,x_{j_b,f'_b}\}$ is an edge in $\h''(\ell'_1,\dots,\ell'_{\sigma}).$ Therefore $b = a$, and hence 
    $E \in \e((\h[\s,\bm x(\s);\al])^\circ).$
\end{proof}

\begin{lem}\label{L-H1isVD}
    With the notation of Set-up ~\ref{SetupMain} and Remark ~\ref{R-HelpRemark},  let $\ell_k = \ell$ for all 
    $k \in [t]$ and $\h_1$ is the hypergraph with the edge set
    $\e(\h_1) = \{E \in \e((\h[\s,\bm x(\s);\al])^\circ): \tilde{E} \cap V = \emptyset\}.$ Then $\h_1$ is vertex decomposable.
\end{lem}

\begin{proof}
    Let $\bm x = (x_{i_q,c_q})_{q = 1}^{\al(\s)}$ and $I = \{i_r \in [m]: r \in [\al(\s)]\}.$ For $i \in I$, let $\mu_i =  \min{C_i}$ and $\nu_i+1 = \max{C_i}$, where
    $C_i = \{c:x_{i,c} = x_{i_r,c_r} ~\mbox{for some}~ r \in [\al]\}.$
    In view of Lemma ~\ref{L-CHSX}, we can write
    $$\bm x(\s) = (x_{1,\mu_1},\dots,x_{1,\nu_1+1},\dots,x_{i,\mu_i},\dots,x_{i,\nu_i+1},\dots).$$
    If $i \in I$, then $x_{i,\mu_i} = x_{i_r,c_r}$ for some $r \in [\al].$ In that case, we set $r' = r+\nu_i-\mu_i+1.$ Then note that $x_{i,\nu_i+1} = x_{i_{r'},c_{r'}}.$
    
    Recall that, for each $i \in [m]$, $\Gamma(x_i) = \Gamma_{i,1}^{(s_{i,1})}\cup \cdots \cup\Gamma_{i,\et_i}^{(s_{i,\et_i})}.$ For $i \in [m]$ and $j \in [\et_i]$, let 
    $\h_{ij} = \h(\Gamma_{i,j}^{(s_{i,j})})$ and 
    \[ \overline{\h}_{ij} = \begin{cases} 
		\left(\h_{ij}(\ell) \setminus (x_{i,\mu_i},\dots,x_{i,\nu_i+1})\right)^\circ & ~\mbox{if}~ i \in I ~\mbox{and}~ \s_{r'} = D ; \\
		\left((\h_{ij}(\ell) \setminus (x_{i,\mu_i},\dots,x_{i,\nu_i}))/ x_{i,\nu_i+1}\right)^\circ & i \in I,~ r < r', ~\mbox{and}~ \s_{r'} = L;\\
        \left(\h_{ij}(\ell) / x_{i,\mu_i}\right)^\circ & ~\mbox{if}~ i \in I ~\mbox{and}~ \s_r = L;\\
        \left(\h_{ij}(\ell) \setminus (x_{i,1},\dots,x_{i,\ell})\right)^\circ & ~\mbox{if}~ i \notin I.
    \end{cases}
	\]
    If $i \in I$, then $x_{i,1},\dots,x_{i,\mu_i-1}$ are isolated vertices in $\h[\s,\bm x(\s);r-1]$ and $x_{i,\nu_i+2},\dots,x_{i,\ell}$ are isolated vertices in $\h[\s,\bm x(\s);r'].$ Thus, we have
     \[ \overline{\h}_{ij} = \begin{cases} 
		\left(\h_{ij}(\ell) \setminus (x_{i,1},\dots,x_{i,\ell})\right)^\circ & ~\mbox{if}~ i \in I ~\mbox{and}~ \s_{r'} = D ; \\
		\left((\h_{ij}(\ell) \setminus (x_{i,1},\dots,x_{i,\nu_i}))/ x_{i,\nu_i+1}\right)^\circ & i \in I,~ r < r', ~\mbox{and}~ \s_{r'} = L;\\
        \left((\h_{ij}(\ell) \setminus (x_{i,1},\dots,x_{i,\mu_i-1}))/ x_{i,\mu_i}\right)^\circ & ~\mbox{if}~ i \in I ~\mbox{and}~ \s_r = L;\\
        \left(\h_{ij}(\ell) \setminus (x_{i,1},\dots,x_{i,\ell})\right)^\circ & ~\mbox{if}~ i \notin I.
    \end{cases}
	\]
    It is easy to see that
    $\h_1 = \bigsqcup_{i = 1}^{m} \bigsqcup_{j = 1}^{\et_i} \overline{\h}_{ij}.$
    Now, it follows from Theorem ~\ref{Th-GSkeleton} that $\h_1$ is vertex decomposable.
\end{proof}
    
\begin{lem}\label{L-Alpha}
    With the notation of the Set-up ~\ref{SetupMain} and Remark ~\ref{R-HelpRemark}, let $\ell_k = \ell$ for all 
    $k \in [t].$ Then $(\h[\s,\bm x(\s);\al])^\circ$ is a vertex decomposable hypergraph.
\end{lem}

\begin{proof}
    Let $E \in \e((\h[\s,\bm x(\s);\al])^\circ).$ We prove that if $\tilde{E} \not\subset V$, then $\tilde{E} \cap V = \emptyset.$
    Using Lemma ~\ref{L-ImConstructible}, there exists $k \in [t]$ such that $\tilde{E} = F_k \setminus W.$ Therefore, $\tilde{E} \cap W = \emptyset.$ Since 
    $\tilde{E} \not\subset V$, it follows from Remark ~\ref{R-MainRemark}(a) that $\tilde{E} \cap V \subset F_k \cap V \subset W.$ Thus, we get 
    $\tilde{E} \cap V = \emptyset.$
    
    Now, we consider the following two cases.
    
    \noindent
    \textbf{Case 1.} When $\{X \in \mathcal{X}\setminus \mathcal{X}_{\max}:X\subset V\} \ne \emptyset.$ In this case, we have $\h' = \h(\mathcal{T}).$ Firstly, we prove that $$(\h[\s,\bm x(\s);\al])^\circ \simeq (\h'[\p,\bm y;\rho])^\circ \sqcup \h_1.$$ 
    Define $\phi:V((\h[\s,\bm x(\s);\al])^\circ) \to V((\h'[\p,\bm y;\rho])^\circ \sqcup \h_1)$ by 
    \[ \phi(x_{j_p,c}) = \begin{cases} 
		x_{j_p,c-c'_p} & ~\mbox{if}~ x_{j_p} \in V; \\
		x_{j_p,c} &  ~\mbox{if}~ x_{j_p} \notin V,
	\end{cases}
	\]
    where $c'_p =\min\{c:x_{j_p,c} \in V((\h[\s,\bm x(\s);\al])^\circ)\}-1.$
    As seen above, if $E \in \e((\h[\s,\bm x(\s);\al])^\circ)$, then either $\tilde{E} \subset V$, or $\tilde{E} \cap V = \emptyset.$ Therefore, $\phi$ is a well-defined map. In view of Lemmas ~\ref{L-Tool} and ~\ref{L-TooLInverse}, $\phi$ is an isomorphism.
    
    \noindent
    \textbf{Case 2.} When $\{X \in \mathcal{X}\setminus \mathcal{X}_{\max}:X\subset V\} = \emptyset.$ In this case, let $\h'' = \h(\mathcal{K}).$ We prove that $$(\h[\s,\bm x(\s);\al])^\circ \simeq \h''(\ell'_1,\dots,\ell'_{\tau}) \sqcup \h_1.$$ Define 
    $\psi:V((\mathcal{H}[\s,\bm x(\s);\al])^\circ) \to V(\overline{\mathcal{H}} \sqcup \h_1)$ given by 
    \[ \psi(x_{j_p,c}) = \begin{cases} 
		x_{j_p,c-c'_p} & ~\mbox{if}~ x_{j_p} \in V; \\
		x_{j_p,c} &  ~\mbox{if}~ x_{j_p} \notin V,
	\end{cases}
	\]
    where $\overline{\h} = \h''(\ell'_1,\dots,\ell'_{\tau})$ and $c'_p = \min\{c:x_{j_p,c} \in V((\mathcal{H}[\s,\bm x(\s);\al])^\circ)\}-1.$ 
    If $E \in \e((\h[\s,\bm x(\s);\al])^\circ)$, then either $\tilde{E} \subset V$, or $\tilde{E} \cap V = \emptyset.$ Therefore, $\psi$ is a well-defined map.   
    In view of Lemma ~\ref{L-2ndTool} $\psi$ is an isomorphism.
    
    In both cases, it follows from Lemma ~\ref{L-H1isVD} and \cite[Lemma 3.20]{BAR} that $\h[\s,\bm x(\s);\al]$ is a vertex decomposable hypergraph.
\end{proof}

\begin{lem}\label{L-MainLemma}
    With notation of Set-up ~\ref{SetupMain}, let $\ell_k = \ell$ for all 
    $k \in [t].$ Then $\h[\s,\bm x(\s);r]$ is a vertex decomposable hypergraph for all $r \ge 0$. In particular, $\h(\ell)$ is a vertex decomposable hypergraph.
\end{lem}

\begin{proof}
    For simplicity, we write $\al = \al(\s)$. Consider the sequence $\bm x(\s)$ obtained in Set-up ~\ref{SetupMain}. In view of  Remark ~\ref{R-MainRemark}(d) and Lemmas ~\ref{L-SheddingV},~\ref{L-Alpha}, $\bm x(\s)$ satisfies the hypothesis of Lemma ~\ref{L-VDtool} for $p = 0.$ Therefore, we obtain the desired result.
\end{proof}

The following example shows that we can't remove condition $(\star)$ on the tuple $(\ell_1,\dots,\ell_t) \in \mathbb{N}^t_{>0}.$ 

\begin{exmp}\label{ESC}
    Let $\Delta$ be a simplicial complex on the vertex set $V = \{x_1,\dots,x_8\}$ with facets 
    $F_1 = \{x_1,x_2,x_3\},F_2 = \{x_3,x_4,x_5\},F_3 = \{x_5,x_6,x_7\},F_4 = \{x_1,x_7,x_8\}$ as shown in Figure ~\ref{Fig-Star1}. Then $W = \{x_1\}$ is a cycle cover of $\Delta.$ Let $\bar{\Delta}$ be a simplicial complex obtained from $\Delta$ by attaching non-pure skeleton complex 
    $\Gamma(x_1) = \langle F_5 = \{x_1,x_9,x_{10}\} \rangle$ as shown in Figure ~\ref{Fig-Star2} and $\h = \h(\bar{\Delta}).$ Since $x_{1,1}$ is the only shedding vertex of $\h(2,1,1,2,1)$ and $\h(2,1,1,,2,1)\setminus x_{1,1}$ has no shedding vertex, it follows that the hypergraph $\h(2,1,1,2,1)$ is not vertex decomposable. 

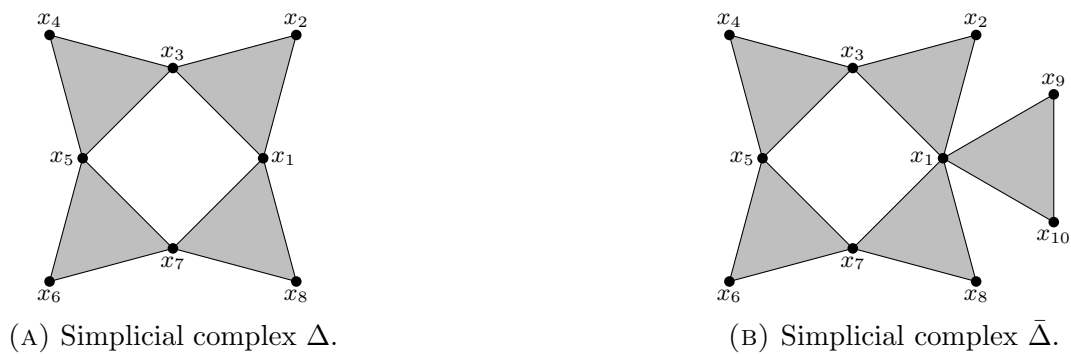
\begin{figure}[H]
    \begin{subfigure}[b]{0.265\textwidth}
    \centering
        \begin{tikzpicture}[node distance=1cm, auto, every node/.style = {inner sep=1pt,outer sep=2pt}, vrtc/.style = {inner sep=3pt}]
		  \fill[lightgray] (1.2,0) -- (1.64,1.64) -- (0,1.2) -- cycle;
		  \fill[lightgray] (-1.2,0) -- (-1.64,1.64) -- (0,1.2) -- cycle;
            \fill[lightgray] (-1.2,0) -- (-1.64,-1.64) -- (0,-1.2) -- cycle;
		  \fill[lightgray] (1.2,0) -- (1.64,-1.64) -- (0,-1.2) -- cycle;
		  \draw (1.2,0) -- (1.64,1.64);
            \draw (1.2,0) -- (0,1.2);
            \draw (1.64,1.64) -- (0,1.2);
		  \draw (-1.2,0) -- (-1.64,1.64);
		  \draw (-1.2,0) -- (0,1.2);
		  \draw (-1.64,1.64) -- (0,1.2);
		  \draw (-1.2,0) -- (-1.64,-1.64);
		  \draw (-1.2,0) -- (0,-1.2);
		  \draw (-1.64,-1.64) -- (0,-1.2);
            \draw (1.2,0) -- (1.64,-1.64);
            \draw (1.2,0) -- (0,-1.2);
            \draw (1.64,-1.64) -- (0,-1.2);
            
		  \begin{scriptsize}
		      \fill  (1.2,0) circle (2.0pt);
		      \draw[right] (1.2,0) node {$x_1$};
		      \fill  (1.64,1.64) circle (2.0pt);
		      \draw[above] (1.64,1.64) node {$x_2$};
		      \fill  (0,1.2) circle (2.0pt);
		      \draw[above] (0,1.2) node {$x_3$};
		      \fill  (-1.64,1.64) circle (2.0pt);
		      \draw[above] (-1.64,1.64) node {$x_4$};
		      \fill  (-1.2,0) circle (2.0pt);
		      \draw[left] (-1.2,0) node {$x_5$};
		      \fill  (-1.64,-1.64) circle (2.0pt);
		      \draw[below] (-1.64,-1.64) node {$x_6$};
                \fill  (0,-1.2) circle (2.0pt);
		      \draw[below] (0,-1.2) node {$x_7$};
                \fill  (1.64,-1.64) circle (2.0pt);
		      \draw[below] (1.64,-1.64) node {$x_8$};
		  \end{scriptsize}
        \end{tikzpicture}
        \caption{Simplicial complex $\Delta$.}
        \label{Fig-Star1}
    \end{subfigure}
    \hspace{4.5cm}
    \begin{subfigure}[b]{0.315\textwidth}
    \centering
    \begin{tikzpicture}[node distance=1cm, auto, every node/.style = {inner sep=1pt,outer sep=2pt}, vrtc/.style = {inner sep=3pt}]
		\fill[lightgray] (1.2,0) -- (1.64,1.64) -- (0,1.2) -- cycle;
	  \fill[lightgray] (-1.2,0) -- (-1.64,1.64) -- (0,1.2) -- cycle;
   	\fill[lightgray] (-1.2,0) -- (-1.64,-1.64) -- (0,-1.2) -- cycle;
	  \fill[lightgray] (1.2,0) -- (1.64,-1.64) -- (0,-1.2) -- cycle;
        \fill[lightgray] (1.2,0) -- (2.672,0.85) -- (2.672,-0.85) -- cycle;
	  \draw (1.2,0) -- (1.64,1.64);
		\draw (1.2,0) -- (0,1.2);
        \draw (1.64,1.64) -- (0,1.2);
        \draw (-1.2,0) -- (-1.64,1.64);
	  \draw (-1.2,0) -- (0,1.2);
		\draw (-1.64,1.64) -- (0,1.2);
	  \draw (-1.2,0) -- (-1.64,-1.64);
		\draw (-1.2,0) -- (0,-1.2);
		\draw (-1.64,-1.64) -- (0,-1.2);
        \draw (1.2,0) -- (1.64,-1.64);
        \draw (1.2,0) -- (0,-1.2);
        \draw (1.64,-1.64) -- (0,-1.2);
        \draw (1.2,0) -- (2.672,0.85);
        \draw (1.2,0) -- (2.672,-0.85);
        \draw (2.672,0.85) -- (2.672,-0.85);
            
		\begin{scriptsize}
		  \fill  (1.2,0) circle (2.0pt);
		  \draw[left] (1.2,0) node {$x_1$};
		  \fill  (1.64,1.64) circle (2.0pt);
		  \draw[above] (1.64,1.64) node {$x_2$};
	   	\fill  (0,1.2) circle (2.0pt);
		  \draw[above] (0,1.2) node {$x_3$};
		  \fill  (-1.64,1.64) circle (2.0pt);
		  \draw[above] (-1.64,1.64) node {$x_4$};
		  \fill  (-1.2,0) circle (2.0pt);
		  \draw[left] (-1.2,0) node {$x_5$};
		  \fill  (-1.64,-1.64) circle (2.0pt);
		  \draw[below] (-1.64,-1.64) node {$x_6$};
            \fill  (0,-1.2) circle (2.0pt);
		  \draw[below] (0,-1.2) node {$x_7$};
            \fill  (1.64,-1.64) circle (2.0pt);
		  \draw[below] (1.64,-1.64) node {$x_8$};
            \fill  (2.672,0.85) circle (2.0pt);
		  \draw[above] (2.672,0.85) node {$x_9$};
            \fill  (2.672,-0.85) circle (2.0pt);
		  \draw[below] (2.672,-0.85) node {$x_{10}$};
		\end{scriptsize}
    \end{tikzpicture}
    \caption{Simplicial complex $\bar{\Delta}$.}
    \label{Fig-Star2}
    \end{subfigure}
    \caption{$\h(\bar{\Delta})(2,1,1,2,1)$ is not vertex decomposable.}
    \label{Fig-Star}
\end{figure}
\end{exmp}

Let $\Delta$ be the simplicial complex as in Example ~\ref{ESC} and $W = \{x_2,x_4,x_6,x_8\}$ be a vertex cover of $\Delta$. Let $\bar{\Delta}$ be the simplicial complex obtained from $\Delta$ by attaching $1$-dimensional simplices at each vertex in $W$. Then it is easy to verify that $\h(\bar{\Delta})(1,1,1,1,1,1,1,1)$ is not a vertex decomposable hypergraph. Thus Lemma ~\ref{L-MainLemma} does not hold if we replace a cycle cover with a vertex cover.

The following theorem is the main result of the paper.

\begin{thm}\label{Th-MainResult}
    Let $\Delta$ be a simplicial complex and $W$ be a cycle cover of $\Delta$. Let $\bar{\Delta}$ be the simplicial complex obtained from $\Delta$ by attaching non-pure skeletons at all vertices of $W$. Then $J(\bar{\Delta})^{(\ell)}$ has linear quotients, and hence it is componentwise linear.
\end{thm}

\begin{proof}
    It follows directly from Lemmas ~\ref{L-VDmeansCL} and ~\ref{L-MainLemma}.
\end{proof}

The following result is an immediate consequence of the Theorem ~\ref{Th-MainResult}.

\begin{corollary}\label{reg} 
    Let $\bar{\Delta}$ be the simplicial complex as in Theorem ~\ref{Th-MainResult}. If $J(\bar{\Delta})^{(\ell)}=J(\bar{\Delta})^{\ell}$ for all $\ell \geq 1$, then $\reg(J(\bar{\Delta})^{(\ell)} = \ell \deg(J(\bar{\Delta})),$  where $\deg(J(\bar{\Delta}))$ denotes the maximum degree of minimal monomial generators of $J(\bar{\Delta}).$ 
\end{corollary}

Now Corollary ~\ref{reg} and Theorem ~\ref{Th-MainResult}, we obtain the following result.

\begin{thm}\label{TH} 
    Let $\bar{\Delta}$ be the simplicial complex as in Theorem ~\ref{Th-MainResult}. If $J(\bar{\Delta})^{(\ell)} = J(\bar{\Delta})^{\ell}$ for all $\ell \geq 1$, then the following are equivalent.
\begin{enumerate}[(a)]
    \item $J(\bar{\Delta})$ has a linear resolution.
    \item $J(\bar{\Delta})^{\ell}$ has a linear resolution for some $\ell \geq 1$.
    \item $J(\bar{\Delta})^{\ell}$ has a linear resolution for all $\ell \geq 1$.
    \item $R/I(\bar{\Delta})$ is Cohen-Macaulay.
    \item $\bar{\Delta}$ is unmixed.
\end{enumerate}
\end{thm}

Let $\Delta$ be a simplicial complex and $W$ be a cycle over of $\Delta.$ Further, let $\bar{\Delta}$ be a simplicial complex obtained from $\Delta$ by attaching a pure skeleton complex at each vertex of $W.$ Then the following example shows that $\h(\bar{\Delta})(\ell)$ need not be vertex decomposable for all 
$\ell \ge 1.$

\begin{exmp}
    Consider the simplicial complex $\Delta = C_4$ as shown in Figure ~\ref{Fig-Pure1}. Clearly, $W = \{x_1\}$ is a cycle cover of $\Delta.$ Let 
    $\bar{\Delta} = \Delta \cup \Gamma^{(2)}$, where $\Gamma$ is the simplex on the vertex set $\{x_1,x_5,x_6,x_7\}.$ 
    \begin{figure}[H]
 	\begin{subfigure}[b]{0.26\textwidth}
 	\centering
 	  \resizebox{\linewidth}{!}{
 		\begin{tikzpicture}[node distance=1cm, auto, every node/.style = {inner sep=1pt,outer sep=2pt}, vrtc/.style = {inner sep=3pt}]
 			\draw (0,0)-- (-1.12,1.12);
 				\draw (-1.12,1.12)-- (-2.24,0);
 				\draw (-2.24,0)-- (-1.12,-1.12);
 				\draw (-1.12,-1.12)-- (0,0);
 					
 				\begin{scriptsize}
						\fill  (0,0) circle (2.0pt);
 					\draw[right] (0,0) node {$x_1$};
 					\fill  (-1.12,1.12) circle (2.0pt);
 					\draw[above] (-1.12,1.12) node {$x_2$};
 					\fill  (-2.24,0) circle (2.0pt);
 					\draw[left] (-2.24,0) node {$x_3$};
 					\fill  (-1.12,-1.12) circle (2.0pt);
						\draw[below] (-1.12,-1.12) node {$x_4$};
 				\end{scriptsize}
 			\end{tikzpicture}
 			}
 			\caption{Simplicial complex $\Delta.$}
                \label{Fig-Pure1}
    \end{subfigure}
    \hspace{0.2\textwidth}
    \begin{subfigure}[b]{0.34\textwidth}
 	\centering
 	\resizebox{\linewidth}{!}{
 		\begin{tikzpicture}[node distance=1cm, auto, every node/.style = {inner sep=1pt,outer sep=2pt}, vrtc/.style = {inner sep=3pt}]
 			\fill[lightgray] (0,0) -- (1.7,1.12) -- (1.7,-1.12) -- cycle;
 				\draw (0,0)-- (-1.12,1.12);
                    \draw (-1.12,1.12)-- (-2.24,0);
 				\draw (-2.24,0)-- (-1.12,-1.12);
 				\draw (-1.12,-1.12)-- (0,0);
 				\draw (0,0) -- (1.7,1.12);
 				\draw (1.7,1.12) -- (1.7,-1.12);
                    \draw (0,0) -- (1.7,-1.12);
 				\draw (1.7,1.12) -- (1,0);
			 	\draw (1.7,-1.12) -- (1,0);
	   			\draw (0,0) -- (1,0);
 						
 				\begin{scriptsize}
 		   		\fill  (0,0) circle (2.0pt);
 					\draw[left] (0,0) node {$x_1$};
 					\fill  (-1.12,1.12) circle (2.0pt);
 					\draw[above] (-1.12,1.12) node {$x_2$};
 					\fill  (-2.24,0) circle (2.0pt);
 					\draw[left] (-2.24,0) node {$x_3$};
 					\fill  (-1.12,-1.12) circle (2.0pt);
						\draw[below] (-1.12,-1.12) node {$x_4$}; 						
                        \fill  (1.7,-1.12) circle (2.0pt);
 					\draw[below] (1.7,-1.12) node {$x_5$};
 					\fill  (1.7,1.12) circle (2.0pt);
 					\draw[above] (1.7,1.12) node {$x_6$};
 					\fill  (1,0) circle (2.0pt);
 					\draw[right] (1,0) node {$x_7$};
 				\end{scriptsize}
 	      \end{tikzpicture}
 		}
 		\caption{Simplicial complex $\bar{\Delta}.$}
 		\end{subfigure}
 		\caption{$\h(\bar{\Delta})(2)$ is not vertex decomposable.}
            \label{Fig-Pure2}
 		\label{Fig-Pure}
    \end{figure}
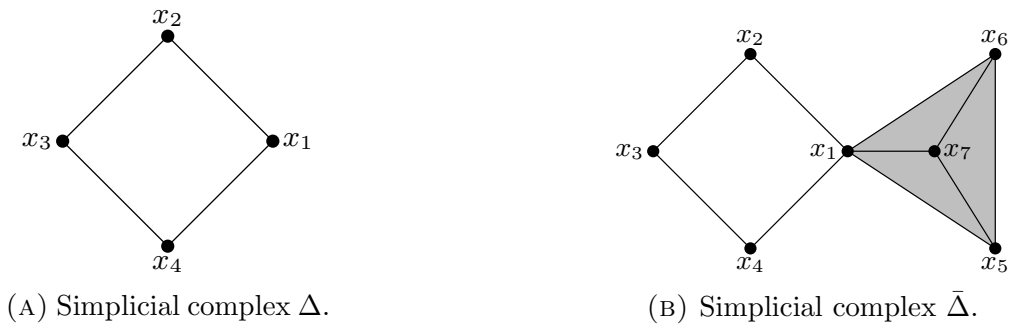
    A computation on Macaulay2, shows that the hypergraph $\h(\bar{\Delta})(2)$ is not vertex decomposable.
\end{exmp}
 
\noindent
{\bf Acknowledgements.}
The second author acknowledges the financial support from the seed grant received from Indian Institute of Technology, Jammu.

\renewcommand{\bibname}{References}

\renewcommand{\bibname}{References}
\bibliographystyle{plain}  

\bibliography{refs_reg}

\end{document}